\documentclass[a4paper,12pt]{amsart}

\usepackage{amssymb, amsmath, amsthm}
\usepackage[foot]{amsaddr}
\usepackage[text={6in,9in},centering]{geometry}
\allowdisplaybreaks[3]

\newcommand{\pD}[2]{\frac{\partial #1}{\partial #2}}
\newcommand{\rD}[2]{\frac{d #1}{d #2}}

\newcommand{\vn}[1]{\lVert#1\rVert}
\newcommand{\IP}[2]{\left< #1 , #2 \right>}

\newcommand{\R}{\ensuremath{\mathbb{R}}}
\newcommand{\N}{\ensuremath{\mathbb{N}}}

\newcommand{\SW}{\ensuremath{\mathcal{W{}}}}

\newcommand{\SH}{\ensuremath{\mathcal{H{}}}}

\newcommand{\BW}{\ensuremath{\text{\bf W}}}

\newcommand{\Cs}{\ensuremath{C_{\!S}}}

\newtheorem{thm}{Theorem}[section]
\newtheorem{cor}[thm]{Corollary}
\newtheorem{prop}[thm]{Proposition}
\newtheorem{lem}[thm]{Lemma}

\theoremstyle{definition}
\newtheorem*{defn}{Definition}
\newtheorem*{rmk}{Remark}

\title[Finite time singularities]{Finite time singularities for the locally constrained Willmore flow of surfaces}

\author{James McCoy$^1$}
\address{$^1$Institute for Mathematics and Applied Statistics\\
         University of Wollongong\\
         Northfields Ave\\
         Wollongong, NSW 2500, Australia}
\author{Glen Wheeler$^2$}
\thanks{Financial support for the second author from the Alexander-von-Humboldt Stiftung is gratefully acknowledged}
\address{$^2$Institut f\"ur Analysis und Numerik\\
Otto-von-Guericke-Universit\"at\\
Postfach 4120\\
D-39016 Magdeburg\\
Germany}
\thanks{\emph{E-mail address:} {\tt wheeler@ovgu.de}}

\begin{document}

\begin{abstract}%{{{

In this paper we study the steepest descent $L^2$-gradient flow of the functional $\SW_{\lambda_1,\lambda_2}$, which is
the the sum of the Willmore energy, $\lambda_1$-weighted surface area, and $\lambda_2$-weighted enclosed volume, for
surfaces immersed in $\R^3$.
This coincides with the Helfrich functional with zero `spontaneous curvature'.
Our first results are a concentration-compactness alternative and interior estimates for the flow.
For initial data with small energy, we prove preservation of embeddedness, and by directly estimating the Euler-Lagrange
operator from below in $L^2$ we obtain that the maximal time of existence is finite.
Combining this result with the analysis of a suitable blowup allows us to show that for such initial data the flow
contracts to a round point in finite time.
\end{abstract}

\maketitle

%}}}

\section{Introduction}%{{{

Suppose we have a surface $\Sigma$ immersed via a smooth immersion $f:\Sigma\rightarrow\R^3$ and consider the functional
\begin{align*}
\SH^{c_0}_{\lambda_1,\lambda_2}(f)
       &= \frac{1}{4}\int_\Sigma (H-c_0)^2d\mu + \lambda_1\mu(\Sigma) + \lambda_2\text{Vol $\Sigma$}.
\end{align*}
In the above we have used $d\mu$ to denote the area element induced by $f$ on $\Sigma$, $d\SH^3$ to denote Hausdorff measure
in $\R^3$, $H$ to denote the mean curvature, $\mu(\Sigma)$ to denote the surface area, $\text{Vol $\Sigma$}$ to denote
the signed enclosed volume, and $c_0,\lambda_1,\lambda_2$ are real numbers.
Our notation is further clarified in Section 2.

Suppose $f_0:\Sigma\rightarrow\R^3$ is an embedded surface.
The Helfrich flow is the steepest descent $L^2$-gradient flow for $\SH^{c_0}_{\lambda_1,\lambda_2}(f)$, and is given by
the one-parameter family of immersions $f:\Sigma\times[0,T)\rightarrow\R^3$ satisfying
\begin{align*}
  \pD{f}{t}  &= -\bigg(\Delta H + H|A^o|^2 + c_0\Big(2K-\frac{1}{2}Hc_0\Big)
                  - 2\lambda_1H - 2\lambda_2\bigg)\nu,
  \\
  f(\cdot,0) &= f_0(\cdot),
\end{align*}
where $\nu$ is the inward pointing unit normal to $f$, $A^o$ denotes the tracefree second fundamental form and $K =
\text{det }A$ denotes the Gauss curvature.
That this flow represents the steepest descent $L^2$-gradient flow for $\SH^{c_0}_{\lambda_1,\lambda_2}(f)$ follows from
its first variation (see Lemma \ref{LMeveq}).

The Helfrich functional is of great interest in applications.
The modern application of the functional to model the shape of an elastic lipid bilayer, such as a biomembrane, is due
to Helfrich \cite{H73}.
Despite the considerable popularity of the functional as a model, there are relatively few analytical results to be
found in the literature.
With $c_0=0$, we have
\begin{align}
\begin{split}
  \pD{f}{t}  &= -\BW_{\lambda_1,\lambda_2}(f)\nu
              = -(\Delta H + H|A^o|^2 - 2\lambda_1H - 2\lambda_2)\nu,
  \\
  f(\cdot,0) &= f_0(\cdot),
\end{split}
\tag{CW}
\label{CW}
\end{align}
which is the steepest descent gradient flow in $L^2$ for the locally constrained Willmore functional
$\SW_{\lambda_1,\lambda_2} = \SH^0_{\lambda_1,\lambda_2}$.
We have used $\BW_{\lambda_1,\lambda_2}$ to denote the Euler-Lagrange operator of $\SW_{\lambda_1,\lambda_2}$.

The above flow, from a physical perspective, corresponds with an assumption that the fluid surrounding the membrane
$f(\Sigma)$ and the fluid contained inside the membrane $f(\Sigma)$ induce zero spontaneous curvature in $f(\Sigma)$.
Thus the flow \eqref{CW} faithfully represents the Helfrich flow in certain settings.
From a more mathematical perspective however, the flow \eqref{CW} is a locally constrained Willmore flow.  (In contrast
with globally constrained flows, such as those considered in \cite{MWW10csdlt,W10slt}.) The normal velocity consists
precisely of a linear combination of the normal velocity of Willmore flow, mean curvature flow, and a constant scaling
factor.

The principal object of study for this paper is the flow \eqref{CW}.
Local existence for \eqref{CW} is explicitly established in \cite{KN06hf} using results of Amann
\cite{A86,A93,Abook,A05}.  We quote the result in the following (weaker) form.

\begin{thm}[Kohsaka-Nagasawa]
Suppose $f_0:\Sigma\rightarrow\R^3$ is a closed immersed surface.
There exists a maximal $T$, $T\in(0,\infty]$, and a corresponding unique one-parameter family
of smooth immersions $f:\Sigma\times[0,T)\rightarrow\R^3$ satisfying \eqref{CW} and $f(\cdot,0) = f_0(\cdot)$.
\label{TMste}
\end{thm}

\begin{rmk}
The evolution equation \eqref{CW} is invariant under tangential diffeomorphisms, and depending on the choice of
$\lambda_1$ and $\lambda_2$ may be also invariant under subgroups of the full M\"obius group of $\R^3$.  (If
$\lambda_1=\lambda_2=0$ then the equation is invariant with respect to the entire M\"obius group.)  The uniqueness in
the local existence theorem above is understood to be modulo these invariances.
\end{rmk}

It is an easy exercise to see that for $\lambda_1>0$ and $\lambda_2\ge 0$ spheres $S_\rho$ shrink self-similarly along
the flow.
It is thus natural to wonder if this property of the flow is robust in the sense that solutions nearby spheres also
shrink in finite time to round points.
It could a priori be the case that there exist local minimisers of the functional in the neighbourhood of spheres, which
prevent the family of spheres $\{S_\rho : \rho\in(0,\infty)\}$ from being local attractors for the flow.
The following classification theorem assures us that this is not the case.

\begin{thm}[{\cite[Theorem 1]{WM12helclass}}]
\label{Tgap}
Suppose $f:\Sigma\rightarrow\R^3$ is a smooth properly immersed surface.
There exists an absolute constant $\varepsilon_1>0$ such that if
\begin{equation}
\int_\Sigma |A^o|^2 d\mu < \varepsilon_1
\label{Egapass}
\end{equation}
then the following statements hold: ($\lambda_1 > 0$)
\begin{align*}
(\lambda_2 < 0)
&\hskip+6mm
 \BW_{\lambda_1,\lambda_2}(f) = 0
\text{ if and only if }
 f(\Sigma) = S_{-\frac{2\lambda_1}{\lambda_2}}(x)\text{ for some $x\in\R^3$},
\\
(\lambda_2 = 0)
&\hskip+6mm
 \BW_{\lambda_1,\lambda_2}(f) = 0
\text{ if and only if }
 f(\Sigma)\text{ is a plane},
\\
(\lambda_2 > 0)
&\hskip+6mm
 \BW_{\lambda_1,\lambda_2}(f) \ne 0.
\intertext{If $\lambda_1 = 0$ then}
(\lambda_2 = 0)
&\hskip+6mm
 \BW_{\lambda_1,\lambda_2}(f) = 0
\text{ if and only if }
 f(\Sigma)\text{ is a plane or a sphere},
\\
(\lambda_2 \ne 0)
&\hskip+6mm
 \BW_{\lambda_1,\lambda_2}(f) \ne 0.
\end{align*}
Here $S_\rho(x) = \partial B_\rho(x)$ denotes the sphere of radius $\rho$ centred at $x\in\R^3$.
\end{thm}

Clearly this implies the following partial result.

\begin{cor}
Suppose $f:\Sigma\times[0,T)\rightarrow\R^3$ is a one-parameter family of closed immersions evolving by \eqref{CW} with
$\lambda_1 > 0$, and $\lambda_2 \ge 0$.
Suppose assumption \eqref{Egapass} is satisfied for each $t\in[0,T)$.
Then $f(\cdot,t)$ is never stationary; that is, for all $p\in\Sigma$ and $t\in[0,T)$ we have
\[
\pD{f}{t}(p,t) = -\BW_{\lambda_1,\lambda_2}\big(f(\cdot,t)\big)\nu(p,t) \ne 0.
\]
\end{cor}

This partial result indicates that a condition such as \eqref{Egapass} on the $L^2$-norm of the tracefree second
fundamental form is appropriate to use as a `distance' from the family of round spheres.
It is not obvious however that if \eqref{Egapass} is initially satisfied, it remains satisfied for all time.
Most importantly, the statement that the flow never reaches a critical point is not anywhere near as strong as stating
that the flow is asymptotic to a shrinking sphere.
It does not even imply that the maximal time of existence is finite.

In this paper we offer the following more comprehensive answer as our main result.

\begin{thm}
Suppose $f:\Sigma\times[0,T)\rightarrow\R^3$ is a one-parameter family of closed immersions evolving by \eqref{CW} with
$\lambda_1 > 0$, $\lambda_2 \ge 0$, and $\text{Vol }\Sigma_0 > 0$.  There exists an $\varepsilon_2 > 0$ depending only
on $\lambda_1$ and $\lambda_2$ such that if 
\begin{equation}
\SW_{\lambda_1,\lambda_2}(f_0) < 4\pi + \varepsilon_2
\label{Eftsingass}
\end{equation}
then
\[
T < 
  \frac1{4\lambda_1^2\pi}
\SW_{\lambda_1,\lambda_2}(f_0) + 1,
\]
and $f(\Sigma,t)$ shrinks to a round point as $t\rightarrow T$.
\label{Tftsing}
\end{thm}

We note that the smallness of $\varepsilon_2$ required may be computed explicitly; it is not the result of a
contradiction argument.

The methods we use in this paper are inspired by recent progress on the analysis of the Willmore functional
\cite{KS01,KS02,KS04} due to Kuwert \& Sch\"atzle.
There are some notable differences between the functional $\SW_{\lambda_1,\lambda_2}$ and the Willmore functional
$\SW_{0,0}$.
The extra terms in $\SW_{\lambda_1,\lambda_2}$ break the conformal invariance of the functional and add to the complexity
of the Euler-Lagrange operator $\BW_{\lambda_1,\lambda_2}$.
Furthermore, for the steepest descent gradient flow of $\SW_{\lambda_1,\lambda_2}$, one loses the a priori global
monotonicity of the Willmore energy.
Indeed, one loses not only the a priori monotonicity of the Willmore energy but also the a priori \emph{uniform bounds}
on the Willmore energy.
Since the flow \eqref{CW} is fourth order and highly non-linear, one should expect that the flow could drive initially
embedded data to a self-intersection.
It is then conceivable that $\text{Vol }\Sigma_t < 0$ for some $t>0$ and therefore one loses all control on
the Willmore energy (and the surface area).
The situation could continue to worsen, with the Willmore energy growing without bound while the energy continues to
satisfy \eqref{Eftsingass}.

Despite these considerations, we show here that the condition \eqref{Eftsingass} is quite suitable for the study of
$\SW_{\lambda_1,\lambda_2}$.
The operator $\BW_{\lambda_1,\lambda_2}$ does not admit a maximum principle, and thus we do not have access to the large
assortment of tools it brings.
We instead rely throughout the paper on estimates for curvature quantities on smooth immersed surfaces combined with the
divergence theorem and the Michael-Simon Sobolev inequality \cite{MS73}.

Our proof of Theorem \ref{Tftsing} relies upon a concentration-compactness alternative (also called a lifespan
theorem), which classifies finite singular times as being local concentrations of the curvature in $L^2$, for a class of
flows larger than those generated by only considering the $L^2$ gradient flow of $\SH^{c_0}_{\lambda_1,\lambda_2}$.
The methods used here are classical interpolation inequalities and energy estimates, such as was used for a large class
of higher order equations in \cite{BG07,G94,G97} and successfully applied to the study of the Willmore flow in \cite{KS02}.

The global analysis of \eqref{CW} requires that one first obtain good control on the Willmore energy and the surface
area along the flow.
As mentioned above, due to the possibility of self-intersections occuring along the flow, we must be quite careful in
using the monotonicity of the energy $\SW_{\lambda_1,\lambda_2}$.
We first prove (cf. \cite{W10sd}) in Proposition \ref{PMonotoneCurvature} that under \eqref{Eftsingass} a conservation
law holds for the Willmore energy, and is itself monotonically decreasing along the flow.
This implies by a well-known result of Li and Yau \cite{LY82ca} that the evolving surface remains embedded for all time.
Using this, we prove $L^1$ estimates for $\vn{A^o}_\infty^4$ (Proposition \ref{PNa04linftyest}), which we then apply to
estimate the $L^2$ norm of $\BW_{\lambda_1,\lambda_2}$ from below.
This immediately gives a quantifiable finite estimate of the extinction time for the flow (Proposition \ref{PNfinite}).
Employing a blowup analysis, we find that the blowup along any blowup sequence is a round sphere.
This implies that the area of the evolving surface vanishes as $t\rightarrow T$ while the surfaces themselves become
asymptotically round.

This paper is organised as follows.
In Section 2 we set up our notation and state the first variation of the functional
$\SH^{c_0}_{\lambda_1,\lambda_2}$.
In Section 3 we establish parabolic regularity theory for a general class of flows.
The main results are the lifespan theorem and the interior estimates, Theorem \ref{Tlt} and Theorem \ref{Tie}
respectively.
Section 4 contains the demonstration of a finite time singularity, including the proof that the maximal existence time
is finite and the blowup classification.
Finally, we included several proofs and derivations of known results in Appendix A for the convenience of the reader.

The authors would each like to thank their home institutions for their support and their collaborator's home
institutions for their hospitality during respective visits.  Both authors would also like to thank Prof.
Graham Williams for useful discussions during the preparation of this work.

%}}} 

\section*{Acknowledgements}%{{{

The research of the first author was supported under the Australian Research Council's Discovery Projects scheme
(project numbers DP0556211 and DP120100097).
The first author is also grateful for the support of the University of Wollongong Faculty of Informatics Research
Development Scheme grant.

Part of this work was carried out while the second author was a research associate supported by the Institute for
Mathematics and Its Applications at the University of Wollongong.
Part of this work was also carried out while the second author was a Humboldt research fellow at the Otto-von-Guericke
Universit\"at Magdeburg.  The support of the Alexander von Humboldt Stiftung is gratefully acknowledged.

%}}}

\section{Preliminaries}%{{{

We consider a surface $\Sigma$ immersed in $\R^3$ via $f:\Sigma\rightarrow\R^3$ and endow a Riemanain metric on $\Sigma$
defined componentwise by
\begin{equation}
g_{ij} = \IP{\partial_if}{\partial_jf},
\label{EQmetric}
\end{equation}
where $\partial$ denotes the regular partial derivative and $\IP{\cdot}{\cdot}$ is the standard Euclidean inner product.
That is, we consider the Riemannian structure on $\Sigma$ induced by $f$, where in particular the metric $g$ is given by
the pullback of the standard Euclidean metric along $f$.  Integration on $\Sigma$ is performed with respect to the
induced area element
\begin{equation}
d\mu = \sqrt{\text{det }g}\ d\SH^3,
\label{EQdmu}
\end{equation}
where $d\SH^3$ is the standard Hausdorff measure on $\R^3$.

The metric induces an inner product structure on all tensor fields defined over $\Sigma$, where corresponding pairs of
indices are contracted.  For example, if $T$ and $S$ are $(1,2)$ tensor fields,
\[
\IP{T}{S}_g = g_{ip}g^{jq}g^{kr}T^{i}_{jk}S^p_{qr},\qquad |T|^2 = \IP{T}{T}_g.
\]
In the above, and in what follows, we shall use the summation convention on repeated indices unless otherwise explicitly
stated.

The \emph{second fundamental form} $A$ is a symmetric $(0,2)$ tensor field over $\Sigma$ with components
\begin{equation}
A_{ij} = \IP{\partial^2_{ij}f}{\nu}.
\label{EQsff}
\end{equation}
where $\nu$ is an inward pointing unit vector field normal along $f$.
With this choice one finds that the second fundamental form of the standard round sphere embedded in $\R^3$ is positive.
There are two invariants of $A$ relevant to our work here: the first is the trace with respect to the metric
\[
H = \text{trace}_g\ A = g^{ij}A_{ij}
\]
called the \emph{mean curvature}, and the second the determinant with respect to the metric, called the \emph{Gauss
curvature},
%%
%%  General definition taken out.
%%
%\[
%K = \text{det}_g\ A = \sum_{\sigma\in R_2}\text{sgn}(\sigma) \prod_{i=1}^2 g^{ik}A_{k\sigma_i},
%\]
%where $R_n$ is the set of all permutations of $\{1,\ldots,n\}$ and sgn is the signature of $\sigma$.  Note the summation
%with respect to $k$ inside the product.
\[
K = \text{det}_g\ A
% = \big(g^{11}A_{11}\big)\big(g^{22}A_{22}\big) - \big(g^{12}A_{12}\big)^2
 = \text{det }\big(g^{ik}A_{kj}\big),
\]
where $\big(P^{ik}Q_{kj}\big)$ is used above to denote the matrix with $i,j$-th component equal to $P^{ik}Q_{kj}$.

The mean and Gauss curvatures are easily expressed in terms of the principal curvatures: at a
single point we may make a local choice of frame for the tangent bundle $T\Sigma$ under which the eigenvalues of $A$
appear along its diagonal.  These are denoted by $k_1$, $k_2$ and are called the \emph{principal curvatures}.  We then
have
\[
H = k_1+k_2,\qquad K = k_1k_2.
\]
We shall often decompose the second fundamental form into its trace and its tracefree parts,
\[
A = A^o + \frac{1}{2}gH,
\]
where $(0,2)$ tensor field $A^o$ is called the \emph{tracefree second fundamental form}.  In a basis which diagonalises
$A$, a so-called principal curvature basis, its norm is given by
\[
|A^o|^2 = \frac{1}{2}(k_1-k_2)^2.
\]
The Christoffel symbols of the induced connection are determined by the metric,
\begin{equation*}
%\label{C3Echristoffelmetric}
\Gamma_{ij}^k = \frac{1}{2}g^{kl}
                \left(\partial_ig_{jl} + \partial_jg_{il} - \partial_lg_{ij}\right),
\end{equation*}
so that then the covariant derivative on $\Sigma$ of a vector $X$ and of a covector $Y$ is
\begin{align*}
\nabla_jX^i &= \partial_jX^i + \Gamma^i_{jk}X^k\text{, and}\\
\nabla_jY_i &= \partial_jY_i - \Gamma^k_{ij}Y_k
\end{align*}
respectively.

From \eqref{EQsff} and the smoothness of $f$ we see that the second fundamental form is
symmetric; less obvious but equally important is the symmetry of the first covariant derivatives of
$A$, 
\[ \nabla_iA_{jk} = \nabla_jA_{ik} = \nabla_kA_{ij}, \]
commonly referred to as the Codazzi equations.

One basic consequence of the Codazzi equations which we shall make use of is that the gradient of the mean curvature is
completely controlled by a contraction of the $(0,3)$ tensor $\nabla A^o$.  To see this, first note that
\[
\nabla_i A^i_j = \nabla_i H = \nabla_i \Big( (A^o)^i_j + \frac12 g_j^i H\Big),
\]
then factorise to find
\begin{equation}
\label{EQbasicgradHgradAo}
\nabla_j H = 2\nabla_i (A^o)^i_j =: 2(\nabla^* A^o)_j.
\end{equation}
This in fact shows that all derivatives of $A$ are controlled by derivatives of $A^o$.
For a $(p,q)$ tensor field $T$, let us denote by $\nabla_{(n)}T$ the tensor field with components
 $\nabla_{i_1\ldots i_n}T_{j_1\ldots j_q}^{k_1\ldots k_p} =
  \nabla_{i_1}\cdots\nabla_{i_n}T_{j_1\ldots j_q}^{k_1\ldots k_p}$.
In our notation, the $i_n$-th covariant derivative is applied first.
Since
\[
\nabla_{(k)} A = \Big(\nabla_{(k)} A^o + \frac12 g \nabla_{(k)} H\Big)
               = \Big(\nabla_{(k)} A^o +         g \nabla_{(k-1)}\nabla^*A^o\Big),
\]
%where we have used the notation $\nabla_{(k)}$ to denote $\overbrace{\nabla \cdots \nabla}^{k\text{ times}}$.
we have
\begin{equation}
\label{EQbasicgradHgradAo2}
|\nabla_{(k)}A|^2 \le 3|\nabla_{(k)}A^o|^2.
\end{equation}
The fundamental relations between components of the Riemann curvature tensor $R_{ijkl}$, the Ricci
tensor $R_{ij}$ and scalar curvature $R$ are given by Gauss' equation
\begin{align*}
R_{ijkl} &= A_{ik}A_{jl} - A_{il}A_{jk},\intertext{with contractions}
g^{jl}R_{ijkl}
   = R_{ik} 
  &= HA_{ik} - A_i^jA_j^k\text{, and}\\
g^{ik}R_{ik}
   = R
  &= |H|^2 - |A|^2.
\end{align*}
We will need to interchange covariant derivatives; for vectors $X$ and covectors $Y$ we obtain
\begin{align}
\notag
\nabla_{ij}X^h - \nabla_{ji}X^h &= R^h_{ijk}X^k = (A_{lj}A_{ik}-A_{lk}A_{ij})g^{hl}X^k,\\
\label{EQinterchangespecific}
\nabla_{ij}Y_k - \nabla_{ji}Y_k &= R_{ijkl}g^{lm}Y_m = (A_{lj}A_{ik}-A_{il}A_{jk})g^{lm}Y_m.
\end{align}
%where $\nabla_{i_1\ldots i_n} = \nabla_{i_1} \cdots \nabla_{i_n}$.
We also use for tensor fields $T$ and $S$ the notation $T*S$ (as in Hamilton \cite{H82}) to denote a linear combination
of new tensors, each formed by contracting pairs of indices from $T$ and $S$ by the metric $g$ with multiplication by a
universal constant.  The resultant tensor will have the same type as the other quantities in the expression it appears.
%Keeping in mind that $\nabla_{(n)}T$ is the tensor field with components $\nabla_{i_1\ldots
%i_n}T_{j_1\ldots}^{k_1\ldots}$
We denote polynomials in the iterated covariant derivatives of $T$ by
\[
P_j^i(T) = \sum_{k_1+\ldots+k_j = i} c_{ij}\nabla_{(k_1)}T*\cdots*\nabla_{(k_j)}T,
\]
where the constants $c_{ij}\in\R$ are absolute.  We use $P_0^0(T)$ to denote a constant.
As is common for the $*$-notation, we slightly abuse this constant when certain subterms do not appear in our $P$-style
terms. For example
\begin{align*}
|\nabla A|^2
    = \IP{\nabla A}{\nabla A}_g
    = 1\cdot\left(\nabla_{(1)}A*\nabla_{(1)}A\right) + 0\cdot\left(A*\nabla_{(2)}A\right)
    = P_2^2(A).
\end{align*}

The Laplacian we will use is the Laplace-Beltrami operator on $\Sigma$, with the components of $\Delta T$
given by
\[
\Delta T_{j_1\ldots j_q}^{k_1\ldots k_p} = g^{pq}\nabla_{pq}T_{j_1\ldots j_q}^{k_1\ldots k_p}  = \nabla^p\nabla_pT_{j_1\ldots j_q}^{k_1\ldots k_p}.
\]
Using the Codazzi equation with the interchange of covariant derivative formula given above, we
obtain Simons' identity:
\begin{align}
g^{kl}\nabla_{ki} A_{lj} &= g^{kl}\nabla_{ik}A_{lj} + g^{kl}g^{pq}R_{kilp}A_{qj} + g^{kl}g^{pq}R_{kijp}A_{lq}
 \notag\\
\Delta A_{ij} &= \nabla_{ik}A_{j}^k
 + g^{kl}g^{pq} (A_{pi}A_{kl}-A_{kp}A_{il}) A_{qj}
 + g^{kl}g^{pq} (A_{pi}A_{kj}-A_{kp}A_{ij}) A_{lq}
 \notag\\
              &= \nabla_{ik}A_{j}^k
 + Hg^{pq}A_{pi}A_{qj}
 - g^{kl}g^{pq} A_{kp}A_{il} A_{qj}
 + g^{kl}g^{pq} A_{pi}A_{kj} A_{lq}
 - A_{ij}\IP{A}{A}_g
 \notag\\
\label{EQsi}  &= \nabla_{ij}H + HA_{i}^kA_{kj} - |A|^2A_{ij},
\end{align}
or in $*$-notation
\[
\Delta A = \nabla_{(2)}H + A*A*A.             
\]
The interchange of covariant derivatives formula for mixed tensor fields $T$ is simple to state in $*$-notation:
\begin{equation}
\label{EQinterchangegeneral}
\nabla_{ij}T = \nabla_{ji}T + T*A*A.
\end{equation}
%In the coming sections we will be concerned with calculating the evolution of the iterated covariant
%derivatives of curvature quantities.
%The following less precise interchange of covariant
%derivatives formula (derived from the fundamental equations above) will be useful to keep in mind:
%\begin{equation}
%\nabla_{ij}T = \nabla_{ji}T + P_2^0(A)*T.
%\label{EQcovint}
%\end{equation}

We now state the first variation of the Helfrich functional for ease of future reference.

\begin{lem}
Suppose $f:\Sigma\rightarrow\R^3$ is a closed immersed surface and $\phi:\Sigma\rightarrow \R^3$ is a vector field
normal along $f$.  Then
\[
\rD{}{t}\SH^{c_0}_{\lambda_1,\lambda_2}(f+t\,\phi)\Big|_{t=0}
 = \frac{1}{2}\int_\Sigma \IP{\phi}{\nu}(\Delta H + H|A^o|^2+2c_0K-(2\lambda_1+c_0^2/2)H-2\lambda_2) d\mu.
\]
In particular, if 
\[
 \phi = \frac{\partial f}{\partial t} = -\big(\Delta H + H|A^o|^2+2c_0K-(2\lambda_1+c_0^2/2)H-2\lambda_2\big)\nu
\]
then
\[
\rD{}{t}\SH^{c_0}_{\lambda_1,\lambda_2}(f)
 = -\frac{1}{2}\int_\Sigma \big|\Delta H + H|A^o|^2+2c_0K-(2\lambda_1+c_0^2/2)H-2\lambda_2\big|^2 d\mu
\]
and the one-parameter family $f:\Sigma\times[0,T)\rightarrow\R^3$ is the steepest descent $L^2$-gradient flow of
$\SH^{c_0}_{\lambda_1,\lambda_2}$.
\label{LMeveq}
\end{lem}
\begin{proof}%{{{
For the proof of the first statement see \cite[Lemma 2.1]{WM12helclass}.
The remaining statements follow from the definition of the $L^2$-gradient.
\end{proof}%}}}

%}}}

\section{Parabolic regularity}%{{{

In this section we first prove that, analagous to the cases of Willmore flow, surface diffusion flow and the constrained
variants thereof \cite{KS02,MWW10csdlt,W10slt,W10sd}, so long as the concentration of curvature remains well-controlled
the flow continues to exist smoothly.
This statement not only holds for $c_0 \ne 0$, but also more generally for flows $f:\Sigma\times[0,T)\rightarrow\R^3$ of
the form
\begin{equation}
\label{E:theflow}
  \frac{\partial f}{\partial t} = -\left( \Delta H + \sum_{\alpha=0}^3 P_\alpha^0(A) \right)\nu.
\end{equation}
The speed $F$ is a second order elliptic differential operator on the Weingarten map $D\nu$, that is, a fourth order
differential operator on $f$.

Our main result in this section is the following concentration-compactness alternative for the class of flows
\eqref{E:theflow}.

\begin{thm}
\label{Tlt}
Let $f: \Sigma \rightarrow \R^3$ be a smooth immersion.
There are absolute constants $\varepsilon_0 > 0$ and $c_0 < \infty$ such that if $\rho > 0$ is chosen with
\begin{equation} \label{E:localAcond}
  \left. \int_{f^{-1}\left( B_{\rho}\left( x \right)\right)} \left| A \right|^{2} d\mu \right|_{t=0} \leq \varepsilon <
\varepsilon_{0}
\end{equation}
for any $x\in \R^3$, then the maximal time $T$ of existence of the flow \eqref{E:theflow} satisfies
\[
T \geq \frac{1}{c_0} \rho^{4} \mbox{,}
\]
and for $0 \leq t \leq \frac{1}{c_0} \rho^{4}$,
\[
\int_{f^{-1}\left( B_{\rho}\left( x \right)\right)} \left| A \right|^{2} d\mu \leq c_0\, \varepsilon_0 \mbox{.}
\]
\end{thm}

\begin{rmk}
It is possible to weaken the regularity requirement on the initial data by exploiting the instantaneous smoothing
property \cite{A10smoothing} of the flow.  In the proof of Theorem \ref{Tlt}, smoothness of the initial data is only
needed to bound the derivatives of curvature at final time. However for this argument we may use in place of the initial
data the immersion at any earlier time: in particular $f_\delta\big(\cdot\big) = f\big(\cdot,\delta\big)$, $\delta \in
(0,T)$, which is smooth (Theorem \ref{T:shorttime}).
\end{rmk}

In the proof of Theorem \ref{Tlt} we shall use local coordinate notation as well as the $*$- and $P$-style notation
introduced in Section 2, which is most convenient for our computations, as for example in \cite{H82}.  We briefly note
that the more general flow \eqref{E:theflow} also enjoys local existence.  The proof is an essentially identical (to
that found in \cite{KN06hf}) verification that the general existence theory of Amann \cite{A86,A93,Abook,A05} applies.
Note again that the uniquness statement below is understood modulo the natural invariances of \eqref{E:theflow}, which
includes at least the family of diffeomorphisms tangential along $f$.

We note that the initial regularity required by Theorem \ref{T:shorttime} below is not optimal.

\begin{thm} \label{T:shorttime}
For any $C^{4,\alpha}$ initial immersion $f_{0} : \Sigma \rightarrow \R^3$, there exists a unique solution
$f:\Sigma\times[0,T)\rightarrow\R^3$ to the flow \eqref{E:theflow} on a maximal time interval $[0,T)$ with initial value
$f_0$ and for which $f_t(\cdot) := f(\cdot,t)$ is smooth for every $t\in(0,T)$.
\end{thm}

The following evolution equations follow from straightforward computations.
Their derivations in a slightly more general setting can be found in Lemma \ref{LMappevo} and Lemma \ref{LMappevoiter}.

\begin{lem} \label{T:evlneqns}
Under the flow \eqref{E:theflow} we have the following evolution equations for various geometric quantities associated
with $f$:
\begin{align*}
\frac{\partial}{\partial t} g_{ij} &=  2 F A_{ij}
\qquad
\frac{\partial}{\partial t} g^{ij} = -2 F A^{ij}
\qquad
\frac{\partial}{\partial t} \nu = - \nabla F
\\
\frac{\partial}{\partial t} \Gamma^{i}_{jk} &= \nabla F * A + FP_1^1(A)
\qquad
\frac{d}{dt} d \mu = - H F d\mu
\\
\frac{\partial}{\partial t} A_{ij} &= - \Delta^{2} A_{ij} + \sum_{\alpha=1}^{3} P_{\alpha}^{2}\left( A \right)+  \sum_{\alpha=2}^{5} P_{\alpha}^{0}\left( A \right)
\\
\frac{\partial}{\partial t} \nabla_{(m)} A_{ij}
 &= - \Delta^{2} \nabla_{(m)} A_{ij}
    + \sum_{\alpha=1}^{3} P_{\alpha}^{m+2}\left( A \right)
    + \sum_{\alpha=2}^{5} P_{\alpha}^{m}\left( A \right).
\end{align*}
\end{lem}

\begin{cor} \label{T:evlncor}
Under the flow \eqref{E:theflow},
\begin{multline*}
  \frac{\partial}{\partial t} \left| \nabla_{(m)} A \right|^{2} = - 2 \nabla^{\beta} \nabla^{\alpha} \nabla_{\alpha} \nabla_{\beta}  \nabla_{i_{1}} \cdots \nabla_{i_{m}} A_{kl} \nabla^{i_{1}} \cdots \nabla^{i_{m}} A^{kl} \\
  + \left[ \sum_{\alpha=1}^{3} P_{\alpha}^{m+2}\left( A \right)+  \sum_{\alpha=2}^{5} P_{\alpha}^{m}\left( A \right)
\right] * \nabla_{(m)}A \mbox{.}
\end{multline*}
\end{cor}

\begin{proof}
This follows by direct computation using Lemma \ref{T:evlneqns} and \eqref{EQinterchangegeneral} as in the proof of
Lemma \ref{LMappevoiter}.
\end{proof}

We now establish energy estimates for the flow.

\begin{lem} \label{T:e1}
Let $\eta: \Sigma \times [0,T] \rightarrow \R$ be a $C^{2}$ function.
While a solution to the flow \eqref{E:theflow} exists,
\begin{align*}
  \rD{}{t} \int_\Sigma \eta \left| \nabla_{(m)} A \right|^{2} d\mu = &- \int_\Sigma  \eta H \Delta H \left| \nabla_{(m)} A
\right|^{2} d\mu + \int_\Sigma  \frac{\partial \eta}{\partial t}  \left| \nabla_{(m)} A \right|^{2} d\mu \\
  & - 2\int_\Sigma \nabla_{\alpha} \nabla_{\beta} \nabla_{i_{1}} \cdots \nabla_{i_{m}} A_{kl} \nabla^{\alpha} \nabla^{\beta} \left( \eta \nabla^{i_{1}} \cdots \nabla^{i_{m}} A^{kl} \right) d\mu \\
  & + \int_\Sigma \eta \left[ \sum_{\alpha=1}^{3} P_{\alpha}^{m+2}\left( A \right)+  \sum_{\alpha=2}^{5}
P_{\alpha}^{m}\left( A \right) \right] * \nabla_{(m)} A \, d\mu \mbox{.}
\end{align*}
\end{lem}

\begin{proof}
The Lemma follows by differentiating using Lemma \ref{T:evlneqns} and Corollary \ref{T:evlncor} and applying the
divergence theorem.
\end{proof}

We shall further specialise by setting $\eta$ to be a smooth cutoff function on the inverse image under $f$ of balls
from $\R^3$.

\begin{defn}
Set $\gamma=\tilde{\gamma}\circ f:\Sigma\rightarrow[0,1]$,
$\tilde{\gamma} \in C^2_c(\R^3)$ satisfying
\begin{equation}
\vn{\nabla\gamma}_\infty \le c_{\gamma},\quad
\vn{\nabla_{(2)}\gamma}_\infty \le c_{\gamma}(c_\gamma+|A|),
\label{Egamma}
\tag{$\gamma$}
\end{equation}
for some absolute constant $c_{\gamma} < \infty$.
\end{defn}

\begin{lem} \label{T:e2}
Suppose $\eta= \gamma^{s}$ where $\gamma$ is as in \eqref{Egamma}, $s\geq 4$ and $\theta > 0$.  While a solution to the flow \eqref{E:theflow} exists,
\begin{align*}
  &\frac{d}{dt} \int_\Sigma  \left| \nabla_{(m)} A \right|^{2} \gamma^{s} d\mu + \left( 2 - \theta \right) \int_\Sigma  \left|
\nabla_{(m+2)} A \right|^{2} \gamma^{s} d\mu \\
  & \quad \leq s \int_\Sigma \left| \nabla_{(m)} A \right|^{2} \gamma^{s-1} \frac{\partial \gamma}{\partial t} d\mu 
  +C \int_\Sigma \left| \nabla_{(m)} A \right|^{2} \gamma^{s-4} \left( \left| \nabla \gamma \right|^{4} + \gamma^{2}
\left| \nabla_{(2)} \gamma \right|^{2} \right) d\mu \\
  &\qquad + \int_\Sigma \gamma^{s} \left[ \sum_{\alpha=1}^{3} P_{\alpha}^{m+2}\left( A \right)+  \sum_{\alpha=2}^{5}
P_{\alpha}^{m}\left( A \right) \right] * \, \nabla_{(m)} A \, d\mu \mbox{,}
\end{align*}
where $C$ is a constant depending only on $\theta$ and $s$.
\end{lem}
\begin{proof}
This follows from Lemma \ref{T:e1} using the divergence theorem and Cauchy's inequality $ab\leq \delta a^{2} +
\frac{1}{4\delta} b^{2}$.
\end{proof}

\begin{lem} \label{T:e3}
Suppose $\gamma$ is as in \eqref{Egamma}, $s\geq 2m+4$ and $\theta>0$.  While a solution to the flow \eqref{E:theflow} exists, we have
\begin{multline*}
  \frac{d}{dt} \int_\Sigma  \left| \nabla_{(m)} A \right|^{2} \gamma^{s} d\mu + \left( 2- \theta \right) \int_\Sigma  \left|
\nabla_{(m+2)} A \right|^{2} \gamma^{s} d\mu \\
\quad  \leq  C\int_\Sigma \left[ \sum_{\alpha=1}^{3} P_{\alpha}^{m+2}\left( A \right)+  \sum_{\alpha=2}^{5}
P_{\alpha}^{m}\left( A \right) \right] * \nabla_{(m)}A\ \gamma^{s} d\mu  + C\int_{\left[ \gamma>0\right]} \left| A \right|^{2} \gamma^{s-4-2m} d\mu \mbox{,}
\end{multline*}
where $C$ is a constant depending only on $\theta$, $s$, $m$ and $c_\gamma$.
\end{lem}
\begin{proof}%{{{
Estimate the time derivative of $\gamma$ by
\[
\pD{\gamma}{t}
 \le c_\gamma|\Delta H + P_3^0(A)| \le c_\gamma P_1^{2}(A) + c_\gamma P_3^0(A),
\]
so that using the divergence theorem and \eqref{Egamma} yields
\begin{align*}
\int_\Sigma &\left| \nabla_{(m)} A \right|^{2} \gamma^{s-1} \frac{\partial \gamma}{\partial t}\, d\mu 
+
\int_\Sigma \left| \nabla_{(m)} A \right|^{2} \gamma^{s-4} (|\nabla\gamma|^4+\gamma^2|\nabla_{(2)}\gamma|^2)\, d\mu 
\\
&\le
  C\int_\Sigma |\nabla_{(m+1)}A|^2\gamma^{s-2}d\mu
+
  C\int_\Sigma |\nabla_{(m)}A|^2\gamma^{s-4}d\mu
\\
&\quad+
  C\int_\Sigma |\nabla_{(m)}A|^2|A|^3\gamma^{s-1}d\mu
+
  C\int_\Sigma |\nabla_{(m)}A|^2|A|^4\gamma^{s}d\mu
\\
&\quad+
  C\int_\Sigma \big(P_3^{m+2}(A)+ P_5^m(A)\big) * \nabla_{(m)}A\ \gamma^{s}d\mu
+
  C\int_{[\gamma>0]} |A|^2 \gamma^{s-4-2m}d\mu,
\end{align*}
where $C$ is a constant depending only on $s$, $m$ and $c_\gamma$.
Young's inequality implies
\[
 \int_\Sigma |\nabla_{(m)}A|^2|A|^3\gamma^{s-1}d\mu
\le
 \frac34
  \int_\Sigma |\nabla_{(m)}A|^2|A|^4\gamma^{s}d\mu
 + \frac14
  \int_\Sigma |\nabla_{(m)}A|^2\gamma^{s-4}d\mu.
\]
Using the divergence theorem it is easy to show that for $\delta>0$ (cf. \eqref{EQintind})
\begin{equation}
\label{EQintiter}
  \int_\Sigma |\nabla_{(m+1)}A|^2\gamma^{s-2}d\mu
\le
  \delta\int_\Sigma |\nabla_{(m+2)}A|^2\gamma^sd\mu
+
  C\int_{[\gamma>0]} |A|^2 \gamma^{s-4-2m}d\mu,
\end{equation}
and similarly
%Iterating \cite[Corollary 5.3]{KS02} twice gives
\[
  \int_\Sigma |\nabla_{(m)}A|^2\gamma^{s-4}d\mu
\le
  \delta\int_\Sigma |\nabla_{(m+2)}A|^2\gamma^sd\mu
+
  C\int_{[\gamma>0]} |A|^2 \gamma^{s-4-2m}d\mu,
\]
where $C$ depends additionally upon $\delta$.
Combining these inequalities we have
\begin{align*}
\int_\Sigma &\left| \nabla_{(m)} A \right|^{2} \gamma^{s-1} \frac{\partial \gamma}{\partial t} d\mu 
+
\int_\Sigma \left| \nabla_{(m)} A \right|^{2} \gamma^{s-4} (|\nabla\gamma|^4+\gamma^2|\nabla_{(2)}\gamma|^2) d\mu 
\\
&\le
  \delta\int_\Sigma |\nabla_{(m+2)}A|^2\gamma^{s}d\mu
+
  C\int_\Sigma \big(P_3^{m+2}(A)+ P_5^m(A)\big) * \nabla_{(m)}A\ \gamma^{s}d\mu
\\&\quad
+
  C\int_{[\gamma>0]} |A|^2 \gamma^{s-4-2m}d\mu,
\end{align*}
which, upon combining with Lemma \ref{T:e2}, finishes the proof.
\end{proof}%}}}

These energy estimates allow us to locally exert strong control on the curvature in $L^2$.

\begin{lem} \label{T:c1}
Let $\gamma$ be as in \eqref{Egamma}.
There exists a constant $\varepsilon$ depending only on $c_\gamma$ such that if
\[
\sup_{0 \leq t <T} \left\| A \right\|^{2}_{2, \left[ \gamma >0 \right]} \leq \varepsilon 
\]
then, under the flow \eqref{E:theflow}, there is a $c$ depending only on $\varepsilon_0$ and $c_\gamma$ such that for any $t\in \left[ 0, T\right)$,
\begin{align}
\int_{\left[ \gamma = 1\right]} \left| A\right|^{2} d\mu
 &+ \int_{0}^{t} \int_{\left[ \gamma=1\right]}
    \left( \left| \nabla_{(2)} A \right|^{2} + \left| A \right|^{2} \left| \nabla A\right|^{2} + \left| A\right|^{6} \right) d\mu \, d\tau
\notag\\
 &\le \left. \int_{\left[ \gamma >0\right]} \left| A \right|^{2} d\mu \right|_{t=0} + c \, \varepsilon \, t \mbox{.}
\label{E:intt}
\end{align}
\end{lem}
\begin{proof}%{{{
Lemma \ref{T:e3} with $s=4$ and $m=0$ gives  
 \begin{multline} \label{E:sm}
  \frac{d}{dt} \int_\Sigma  \left| A \right|^{2} \gamma^{4} d\mu + \left( 2-\theta \right) \int_\Sigma  \left|
\nabla_{(2)} A \right|^{2} \gamma^{4} d\mu \\
 \leq  \int_\Sigma \left[ \sum_{\alpha=1}^{3} P_{\alpha}^{2}\left( A \right)+  \sum_{\alpha=2}^{5} P_{\alpha}^{0}\left( A
\right) \right] * A \gamma^{4} d\mu  + c\int_{\left[ \gamma>0\right]} \left| A \right|^{2} d\mu \mbox{.} 
\end{multline}
We estimate
\[
\sum_{\alpha = 1}^{3} P_{\alpha}^{2} \left( A\right) * A \leq 
  c \left[ \left( 1 +  \left| A \right|^{2} \right)  \left| \nabla_{(2)}A \right| + \left( 1 +  \left| A \right| \right)  \left| \nabla A \right|^{2} \right] \left| A \right|
\]
and
\[
\sum_{\alpha = 2}^{5} P_{\alpha}^{0} \left( A\right) * A \leq c \left( \left| A \right|^{3} +  \left| A \right|^{4} + \left| A \right|^{5} +  \left| A \right|^{6} \right) \mbox{.}
\]
Therefore
\[
\int_\Sigma \sum_{\alpha = 2}^{5} P_{\alpha}^{0} \left( A\right) * A \gamma^4 d\mu\leq c \int_{\left[ \gamma >0\right]}
\left| A \right|^{2} d\mu + c \int_\Sigma \left| A \right|^6 \gamma^4 d\mu \mbox{.}
\]
We also estimate for $\delta > 0$
\[
\int_\Sigma |\nabla_{(2)}A| \left| A \right| \gamma^{4} d\mu \leq \delta \int_\Sigma \left| \nabla_{(2)} A \right|^{2} \gamma^{4} d\mu + \frac{1}{4\delta} \int_{\left[ \gamma>0\right]} \left| A\right|^{2} d\mu
\]
and
\[
\int_\Sigma | \nabla_{(2)}A | \left| A \right|^{3} \gamma^{4} d\mu \leq \delta \int_\Sigma \left| \nabla_{(2)}
A \right|^{2} \gamma^{4} d\mu + \frac{1}{4\delta} \int_\Sigma \left| A\right|^{6} \gamma^{4} d\mu \mbox{.}
\]
The last term on the right is now estimated exactly as in \cite{KS02} using several applications of the Michael-Simon
Sobolev inequality \cite{MS73}: For any $u\in C^{1}_{c}\left( \Sigma \right)$,
\begin{equation}
\left( \int_\Sigma u^{2} d\mu \right)^{\frac{1}{2}} \leq \frac{4^3}{\omega_2^{1/2}}\left( \int_\Sigma \left| \nabla u\right| d\mu +
\int_\Sigma \left| H \right| \left| u \right| d\mu \right) \mbox{,}
\label{EQmss}
\end{equation}
where $\omega_n=\SH^n(B_1)$.
Kuwert and Sch\"{a}tzle used this to establish the
following flow-independent inequality for immersed surfaces \cite[Lemma 4.2]{KS02}: 
\begin{multline} \label{E:4.2}
  \int_\Sigma \left| A \right|^6 \gamma^s d\mu + \int_\Sigma \left| A\right|^2 \left| \nabla A \right|^2 \gamma^s d\mu \\
    \leq c\int_{\left[ \gamma > 0\right]} \left| A\right|^2 d\mu \int_\Sigma \left( \left| \nabla_{(2)} A
\right|^2 + \left| A \right|^6 \right) \gamma^s d\mu + c\left( \int_{\left[ \gamma > 0\right]} \left| A\right|^2 d\mu \right)^2 \mbox{,}
\end{multline}
where $c$ is a constant depending only on $s$ and $c_\gamma$.
We additionally need
\[
\int_\Sigma \left| A\right| \left| \nabla A \right|^2 \gamma^4 d\mu \leq \frac{1}{2} \left( \int_\Sigma \left| \nabla
A \right|^2 \gamma^4 d\mu + \int_\Sigma \left| A \right|^2 \left| \nabla A \right|^2 \gamma^4 d\mu \right)
\]
and by integration by parts 
\[
 \int_\Sigma \left| \nabla A \right|^2 \gamma^4 d\mu \leq \delta \int_\Sigma \left| \nabla_{(2)} A \right|^2 \gamma^4 d\mu
+ c \int_{\left[ \gamma >0\right]} \left| A \right|^{2} d\mu,
\]
where $c$ is a constant depending only on $c_\gamma$ and $\delta$.
Applying \eqref{E:4.2},
\begin{multline*}
  \int_\Sigma \left| A\right| \left| \nabla A \right|^2 \gamma^4 d\mu \leq \delta \int_\Sigma \left| \nabla_{(2)} A \right|^2 \gamma^4 d\mu + c \int_{\left[ \gamma >0\right]} \left| A \right|^{2} d\mu\\
  +  c\int_{\left[ \gamma > 0\right]} \left| A\right|^2 d\mu \int_\Sigma \left( \left| \nabla_{(2)} A
\right|^2 + \left| A \right|^6 \right) \gamma^s d\mu + c \left( \int_{\left[ \gamma > 0\right]} \left| A\right|^2 d\mu \right)^2 \mbox{.}
  \end{multline*}
Altogether we have
\begin{align*}
  & \frac{d}{dt} \int_\Sigma  \left| A \right|^{2} \gamma^{4} d\mu + \left( 2 - \theta \right) \int_\Sigma  \left|
\nabla_{(2)} A \right|^{2} \gamma^{4} d\mu \\
  & \leq c \int_{\left[ \gamma > 0\right]} \left| A\right|^2 d\mu  \int_\Sigma \left( \left| \nabla_{(2)} A \right|^{2} + \left| A \right|^{6} \right) \gamma^{4} d\mu 
 + c \left( \int_{\left[ \gamma > 0\right]} \left| A\right|^2 d\mu \right)^2 + c \int_{\left[ \gamma > 0\right]} \left| A\right|^2 d\mu  \mbox{.}
   \end{align*}
Using \eqref{E:4.2} again,
\begin{align*}
   &\frac{d}{dt} \int_\Sigma  \left| A \right|^{2} \gamma^{4} d\mu + \left( 2 - \theta \right) \int_\Sigma  \left(
\left| \nabla_{(2)} A \right|^{2} + \left| A\right|^{2} + \left| A\right|^6 \right) \gamma^{4} d\mu \\
 &  \leq c \int_{\left[ \gamma > 0\right]} \left| A\right|^2 d\mu  \int_\Sigma \left( \left| \nabla_{(2)} A \right|^{2} + \left| A \right|^{6} \right) \gamma^{4} d\mu
 + c \left( \int_{\left[ \gamma > 0\right]} \left| A\right|^2 d\mu \right)^2 + c \int_{\left[ \gamma > 0\right]} \left| A\right|^2 d\mu \\
 & \leq \varepsilon \int_\Sigma \left( \left| \nabla_{(2)} A \right|^{2} + \left| A \right|^{6} \right) \gamma^{4} d\mu + c\left( \varepsilon +1 \right) \varepsilon \mbox{.}
    \end{align*}
Therefore for $\theta, \varepsilon$ small enough there is a $c$ depending only on $\varepsilon$ and $c_\gamma$ such that
   $$\frac{d}{dt} \int_\Sigma  \left| A \right|^{2} \gamma^{4} d\mu + \int_\Sigma  \left( \left| \nabla_{(2)} A
\right|^{2} + \left| A\right|^{2}|\nabla A|^2 + \left| A\right|^6 \right) \gamma^{4} d\mu \leq c \, \varepsilon \mbox{.}$$
The result now follows by integrating in time.
\end{proof}%}}}

\begin{lem} \label{T:c2}
Suppose $\gamma$ is as in \eqref{Egamma} and $s\geq 2m+4$.  While a solution to the flow \eqref{E:theflow} exists, we
have
\begin{multline*}
 \frac{d}{dt} \int_\Sigma  \left| \nabla_{(m)} A \right|^{2} \gamma^{s} d\mu + \int_\Sigma  \left| \nabla_{(m+2)} A
\right|^{2} \gamma^{s} d\mu \\
 \leq  C \left\| A \right\|^{4}_{\infty, \left[ \gamma >0\right]}\int_\Sigma \left| \nabla_{(m)} A \right|^{2}
\gamma^{s} d\mu 
 +  C \left(1 + \left\| A \right\|^{4}_{\infty, \left[ \gamma >0\right]}\right)  \left(
\left\| A \right\|^{2}_{2, \left[ \gamma>0 \right]} \right) \mbox{,}
 \end{multline*}
where $C$ is a constant depending only on $s$, $m$ and $c_\gamma$.
\end{lem}
\begin{proof}%{{{
We use with Lemma \ref{T:e3} the following interpolation inequalities (see \cite[Appendix]{KS02}) for tensor fields $T$ on $\Sigma$:
\begin{enumerate}
  \item[(i)] Let $0 \leq i_{1}, \ldots, i_{r} \leq k$, $i_{1} + \cdots + i_{r} =2k$ and $s\geq 2k$.  Then for a constant
$C$ depending only on $k$, $s$, $r$ and $c_\gamma$,
\begin{equation} \label{E:int1}
  \int_\Sigma \left| \nabla_{(i_{1})} T * \cdots * \nabla_{(i_{r})} T \right| \gamma^{s} d\mu 
    \leq C \left\| T \right\|^{r-2}_{\infty, \left[ \gamma>0\right]} \left( \int_\Sigma \left| \nabla_{(k)} T \right|^{2}
\gamma^{s} d\mu + \left\| T \right\|^{2}_{2, \left[ \gamma >0\right]} \right) \mbox{.}
\end{equation}
  \item[(ii)] Let $1\leq p, q, r \leq \infty$ satisfy $\frac{1}{p} + \frac{1}{q} = \frac{1}{r}$ and let $\alpha, \beta \geq 0$ 
satisfy $\alpha + \beta = 1$.  Then for $s\geq \max \left( \alpha q, \beta p \right)$ and $-\frac{1}{p} \leq t \leq
\frac{1}{q}$, there is a constant $C$ depending only on $r$ such that
\begin{multline} \label{E:int2}
  \left( \int_\Sigma \left| \nabla T \right|^{2r} \gamma^{s} d\mu \right)^{\frac{1}{r}} \leq C\bigg[ \left( \int_\Sigma \left| T
\right|^{q} \gamma^{s\left( 1-tq\right)} d\mu\right)^{\frac{1}{q}} \left( \int_\Sigma \left| \nabla_{(2)} T \right|^{p}
\gamma^{s\left( 1+t p \right)} d\mu \right)^{\frac{1}{p}} \\*
  +  (c_\gamma) s \left( \int_\Sigma \left| T \right|^{q} \gamma^{s - \alpha q } d\mu \right)^{\frac{1}{q}} \left(
\int_\Sigma
\left| \nabla T \right|^{p} \gamma^{s - \beta p } d\mu \right)^{\frac{1}{p}}\bigg] \mbox{.}
  \end{multline}
\end{enumerate}
A straightforward proof by induction on \eqref{E:int2} above yields additionally
\begin{equation}
\label{EQintind}
\left( \int_\Sigma |{\nabla_{(k)}A}|^p\gamma^sd\mu \right)^\frac{1}{p}
\\
  \le \delta \left( \int_\Sigma |\nabla_{(k+1)}A|^p\gamma^{s+p}d\mu \right)^\frac{1}{p}
      + C \left( \int_{[\gamma>0]}|A|^p\gamma^{s-kp}d\mu \right)^\frac{1}{p},
\end{equation}
where $\delta > 0$, $2\le p < \infty$, $k\in\N$, $s\ge k\,p$, and $C$ is a constant depending only on $\delta$ and $c_\gamma$.

From \eqref{E:int1} we have for each $\alpha= 2, \ldots, 5$,
$$  \int_\Sigma P_{\alpha}^{m}\left( A \right) * \nabla_{(m)} A \gamma^{s} d\mu \leq C \left\| A \right\|_{\infty, \left[
\gamma >0\right]}^{\alpha-1} \left( \int_\Sigma \left| \nabla_{(m)} A \right|^{2} \gamma^{s} d\mu + \left\| A \right\|_{2, \left[ \gamma >0\right]}^{2} \right)$$
where $C$ is a constant depending only on $\alpha$, $m$, $s$  and $c_\gamma$.
Thus
\begin{align*}
  &\int_\Sigma \sum_{\alpha=2}^{5} P_{\alpha}^{m}\left( A\right) * \nabla_{(m)}A \gamma^{s} d\mu \\
  & \quad \leq 
   c \sum_{\alpha=1}^{4} \left\| A \right\|_{\infty, \left[ \gamma>0\right]}^{\alpha} \left( \int_\Sigma \left|
\nabla_{(m)} A \right|^{2} \gamma^{s} d\mu + \left\| A \right\|_{2, \left[ \gamma>0\right]}^{2} \right)\\
  & \quad \leq c \left( 1 + \left\| A \right\|_{\infty, \left[ \gamma>0\right]}^{4} \right) \left( \int_\Sigma \left|
\nabla_{(m)} A \right|^{2} \gamma^{s} d\mu + \left\| A \right\|_{2, \left[ \gamma>0\right]}^{2} \right) \mbox{.}
   \end{align*}
Again using \eqref{E:int1}, Cauchy's inequality and the divergence theorem
\begin{align*}
 &  \int_\Sigma \sum_{\alpha=1}^{3} P_{\alpha}^{m+2}\left( A \right) * \nabla_{(m)} A \gamma^{s} d\mu \\
 & \quad \leq C \left( \left\| A \right\|_{\infty, \left[ \gamma >0\right]}^{2} + \left\| A \right\|_{\infty, \left[
\gamma >0\right]} \right) \left( \int_\Sigma \left| \nabla_{(m+1)} A \right|^{2} \gamma^{s} d\mu + \left\| A
\right\|_{2, \left[ \gamma >0\right]}^2 \right) \\
& \qquad   +  \delta \int_\Sigma \left| \nabla_{(m+2)} A \right|^{2} \gamma^{s} d\mu + C\left(  \delta
\right) \left(1 + \left\| A \right\|_{\infty, \left[ \gamma >0\right]}^{4}\right) \int_\Sigma \left| \nabla_{(m)} A \right|^{2} \gamma^{s} d\mu \mbox{.}
   \end{align*}
Estimating the $\vn{\nabla_{(m+1)}A}_2^2$ and $\vn{\nabla_{(m)}A}_2^2$ terms using \eqref{EQintiter}, \eqref{E:int2}, \eqref{EQintind},  then combining our
estimates completes the proof.
\end{proof}%}}}

\begin{lem} \label{T:c3}
Let $\gamma$ be as in \eqref{Egamma}. Under the flow \eqref{E:theflow}, if
\[
\sup_{0\leq t\leq T} \int_{\left[ \gamma >0\right]} \left| A\right|^{2} d\mu \leq \varepsilon
\]
where $\varepsilon$ is a constant depending only on $c_\gamma$, then
\begin{equation} \label{E:higher}
  \left\| \nabla_{(m)} A \right\|_{\infty, \left[ \gamma =1\right]} \leq c
\end{equation}
where $c$ is a constant depending only on $m$, $T$, $c_\gamma$, and $\alpha_{0}(m+2)$, where
\[
\alpha_{0}\left( m \right) = \sum_{j=0}^{m} \left\| \nabla_{(j)} A\right\|_{2, \left[ \gamma >0 \right]} \Big|_{t=0}.
\]
\end{lem}
%Note that $\alpha_{0}$ depends only on the initial surface $M_{0}$ and the choice of $\gamma$.\\
\begin{proof}%{{{
This is similar to the proof of Proposition 4.6 in \cite{KS02}, using Lemma \ref{T:c1}, Lemma \ref{T:c2} and the
same argument based on the Michael-Simon Sobolev inequality.  We provide a sketch for completeness.  In particular, we
will use from \cite{KS02}: For any tensor $T$ on $\Sigma$ and $\gamma$ as in \eqref{Egamma},
\begin{equation} \label{E:T}
  \left\| T \right\|_{\infty, \left[ \gamma=1\right]}^4 \leq c   \left\| T \right\|_{2, \left[ \gamma>0\right]}^2 \left(
\left\| \nabla_{(2)} T \right\|_{2, \left[ \gamma>0\right]}^2 + \lVert \left| T \right| \left| A\right|^2 \rVert_{2,
\left[ \gamma>0\right]}^2 +  \left\| T \right\|_{2, \left[ \gamma>0\right]}^2 \right) \mbox{.}
\end{equation}
Further, if $T=A$ and $\left\| A \right\|_{2, \left[ \gamma>0\right]}^2 \leq \varepsilon$, for some small
$\varepsilon$ depending only on $c_\gamma$, then together with \eqref{E:4.2} and a trivial covering argument we obtain
\begin{equation} \label{E:T2}
  \left\| A \right\|_{\infty, \left[ \gamma=1\right]}^4 \leq c   \left\| A \right\|_{2, \left[ \gamma>0\right]}^2 \left(
\left\| \nabla_{(2)} A \right\|_{2, \left[ \gamma>0\right]}^2 +  \left\| A \right\|_{2, \left[ \gamma>0\right]}^2 \right) \mbox{.}
  \end{equation}
  % (covering argument to absorb \int A^6 on RHS to get \int on a smaller ball)
For a given choice of cutoff function $\gamma$, for $0\leq \sigma < \tau \leq 1$ set $\gamma_{\sigma, \tau} = \psi_{\sigma, \tau} \circ \gamma$ where
$$\gamma_{\sigma, \tau} = \begin{cases}
 0 \mbox{ for } \gamma \leq \sigma\\
 1 \mbox{ for } \gamma \geq \tau \mbox{.} \end{cases}$$
Choose $\psi_{\sigma, \tau} $ such that bounds of the form in \eqref{Egamma} hold.
From Lemma \ref{T:c1}, with $\sigma=0$ and $\tau= \frac{1}{2}$,
\begin{equation}
\label{EQie1} \int_{0}^{T} \int_{\left[\gamma \geq \frac{1}{2}\right]} \left( \left| \nabla_{(2)} A \right|^2 + \left| A
\right|^6 \right) d\mu\, d\tau \leq c \varepsilon \left( 1 + T\right) \mbox{.}
\end{equation}
Now using $\gamma_{\frac{1}{2}, \frac{3}{4}}$ in \eqref{E:T2},
\begin{equation}
\label{EQie2}
\int_{0}^{T} \left\| A\right\|^{4}_{\infty, \left[ \gamma \geq \frac{3}{4} \right]} d\tau \leq c\,\varepsilon^{2}\left( 1 + T\right) \mbox{.}
\end{equation}
With $\sigma=\frac{3}{4}$ and $\tau = \frac{7}{8}$ we obtain from Lemma \ref{T:c2}
\begin{align}
 & \int_\Sigma \left| \nabla_{(m)} A\right|^2 \gamma_{\sigma, \tau}^{s} d\mu + \int_0^t \int_{\left[ \gamma \geq \frac{7}{8}
\right]} \left| \nabla_{(m+2)} A\right|^2 d\mu\, d\tau
\notag
\\
  &\quad \leq \left.  \int_\Sigma \left| \nabla_{(m)} A\right|^2 \gamma_{\sigma, \tau}^{s} d\mu \right|_{t=0} + c
\label{EQie6}
\varepsilon \left( T + \int_0^T \vn{A}^4_{\infty, \left[ \gamma \geq\frac{3}{4} \right]} d\mu\,d\tau \right)
\\
\notag
  &\qquad + c\int_0^t \left(  \vn{A}^4_{\infty, \left[ \gamma \geq\frac{3}{4} \right]} \int_\Sigma \left| \nabla_{(m)} A\right|^2 \gamma_{\sigma, \tau}^{s} d\mu \right) d\tau \mbox{.}
\end{align}
In view of \eqref{EQie2}, applying Gronwall's inequality gives
$$\int_\Sigma \left| \nabla_{(m)} A\right|^2 \gamma_{\sigma, \tau}^{s} d\mu\leq c_m \mbox{,}$$
where here and throughout the proof $c_m$ is a constant depending only on $\alpha_0(m)$ and $T$.
Using this in \eqref{EQie6} we obtain
$$\int_0^t \int_{\left[ \gamma \geq \frac{7}{8} \right]} \left| \nabla_{(m+2)} A\right|^2 d\mu\, d\tau \leq c_m \mbox{.}$$
Hence
$$\sup_{\left[ 0, T\right]} \int_{\left[ \gamma \geq \frac{7}{8} \right]} \left| \nabla_{(m)} A\right|^2  d\mu + \int_0^T
\int_{\left[ \gamma \geq \frac{7}{8} \right]} \left| \nabla_{(m+2)} A\right|^2 d\mu\, d\tau \leq c_m \mbox{.}$$
Now from \eqref{E:T2}
$$\left\| A \right\|^4_{\infty, \left[ \gamma \geq \frac{15}{16} \right]} \leq \varepsilon\,c_2  $$
and using \eqref{E:T} with $T=\nabla_{(m)} A$ we find
\begin{align*}
  & \left\| \nabla_{(m)} A \right\|^4_{\infty, \left[ \gamma =1 \right]} \\
  & \quad \leq c \left\| \nabla_{(m)} A \right\|^2_{2, \left[ \gamma \geq \frac{15}{16} \right]}
   \left( \left\| \nabla_{(m+2)} A \right\|^2_{2, \left[ \gamma \geq \frac{15}{16} \right]} \right. \left. +
\left\| A^2 * \nabla_{(m)} A \right\|^2_{2, \left[ \gamma \geq \frac{15}{16} \right]} + \left\| \nabla_{(m)} A \right\|^2_{2, \left[ \gamma \geq \frac{15}{16} \right]} \right)\\
  & \quad \leq c_m \left[ c_{m+2} + \left( c_2 + 1\right) c_m \right]
\end{align*}
completing the proof.
\end{proof}%}}}

\begin{proof}[Proof of Theorem \ref{Tlt}]%{{{
Given the bounds of Lemma \ref{T:c3}, this is essentially the same proof by contradiction to the maximality of $T$ as in
\cite{KS02}, using the result on equivalent metrics in \cite{H82}.  The only differences that arise are the result of
the extra terms in the evolution equation for the more general flow \eqref{E:theflow} and subsequent additional terms in
Lemma \ref{T:evlneqns}, but these are controlled using Lemma \ref{T:c3}.  For completeness, we provide a sketch of the
proof.

We may assume by rescaling $f\left( x, t\right) \mapsto \frac{1}{\rho} f\left( x, \rho^4 t\right)$ that $\rho=1$, and
thus need to show $T\geq \frac{1}{c}$.  Set 
 $$\eta\left( t\right) = \sup_{x\in \mathbb{R}^{3}} \int_{f^{-1}\left( B_1\left( x\right) \right)} \left| A\right|^{2} d\mu \mbox{.}$$  
 Via short time existence, $f\left( M \times \left[ 0, t\right] \right)$ is compact for any $t<T$ and $\eta\left( t\right)$ is continuous.  
 % approximate sup with a limit, inside is bounded independent of x and t so can use Fatou's Lemma (check) to interchange limits
 Observe
  $$\eta\left( t\right) \leq c_\eta \sup_{x\in \mathbb{R}^{3}} \int_{f^{-1}\big( B_\frac{1}{2}\left( x\right) \big)} \left| A\right|^{2} d\mu \mbox{.}$$
Set, for $\lambda$ to be chosen,
  $$t_0=\sup\left\{ 0\leq t\leq \min\left( T, \lambda \right): \eta\left( \tau\right) \leq 3 \, c_\eta \, \varepsilon_0 \mbox{ for } 0 \leq \tau < t \right\} \mbox{.}$$
It can be shown using \eqref{E:intt} that with $\lambda = \frac{1}{c}{c_\eta}$ and provided $\varepsilon_0$ is small
enough, $t_0 = \min \left( T, \lambda \right)$.  Small enough $\varepsilon_0$ may be obtained by taking $\rho$ small
enough in \eqref{E:localAcond}.

Now if $t_0 = \lambda$ we are done, since then $T\geq \lambda=\frac{1}{c}$ and the integral estimate follows from
\eqref{E:intt}.  So it remains to show we cannot have $t_0 = T < \infty$ by contradicting the maximality of $T$.  (If
$T=\infty$ the result trivially holds.)  So suppose for the sake of obtaining a contradiction that $t_0=T$.  From
\eqref{E:higher} we have
$$\left\| \nabla_{(m)} A \right\|_{\infty} \leq c\left( m, T, \alpha_{0}\left( m+2 \right) \right) \mbox{.}$$

A result of Hamilton in \cite{H82} implies the metrics on $\Sigma_t$ are all uniformly equivalent for $0 \leq t \leq T$.  Converting \eqref{E:higher} into bounds on parameter derivatives of $f$ we have
$$\left\| \partial^m f \right\|_{\infty}, \left\|\partial^m \pD{}{t} f \right\|_{\infty} \leq c\left( m, f_0 \right)$$
where the $\left\| f \right\|_{\infty}$ bound for finite time $T$ follows from \eqref{E:theflow} and \eqref{E:higher}
with $m=0, 1, 2$.  So $f \left( \cdot, t \right) \rightarrow f \left( \cdot, T\right)$ in $C^{\infty}$ and $\Sigma_T$ is smooth.  
%Now $f \left(\cdot, T\right)$ is a smooth immersion via the uniform equivalent metrics and $g\left( t\right) \rightarrow g\left( T\right)$.  
This then allows extension of the solution using short time existence, contradicting the maximality of $T$.
\end{proof}%}}}

To ensure the existence of a smooth blowup we also need the following version of Lemma \ref{T:c3} which is localised in time
(cf. \cite[Theorem 3.5]{KS01}).

\begin{thm}
\label{Tie}
Suppose $f:\Sigma\times(0,\delta]\rightarrow\R^3$ flows by \eqref{E:theflow} and satisfies
\[
\sup_{0<t\le\delta} \int_{f^{-1}(B_{2\rho}(0))} |A|^2 d\mu \le \varepsilon < \varepsilon_0,
\]
where $\delta \le c\rho^4$.  Then for any $k\in\N_0$ and $t\in(0,\delta)$ we have
\begin{align*}
\vn{\nabla_{(k)}A}_{2,f^{-1}(B_\rho(0))} &\le c_k\sqrt{\epsilon}t^{-\frac{k}{4}}
\\
\vn{\nabla_{(k)}A}_{\infty,f^{-1}(B_\rho(0))} &\le c_k\sqrt{\epsilon}t^{-\frac{k+1}{4}}
\end{align*}
where $c_k$ is an absolute constant for each $k$.
\end{thm}
\begin{proof}%{{{
By scaling, we may assume $\rho=1$.  In this proof we shall abbreviate $B_\rho(0)$ with $B_\rho$.  Estimates
\eqref{EQie1}, \eqref{EQie2} imply
\begin{equation}
\label{EQie3}
\int_0^\delta \int_{f^{-1}(B_{\frac34})} \big(|\nabla_{(2)}A|^2 + |A|^6\big)\, d\mu\, d\tau
+
\int_0^\delta \vn{A}^4_{\infty,f^{-1}(B_{\frac34})} d\tau
 \le c\,\varepsilon,
\end{equation}
where $c$ depends on $\delta$.
Let $t^*<\delta$.
Consider piecewise linear cutoff functions in time $\chi_j:[0,t^*]\rightarrow[0,1]$ defined by
\[
\chi_j(t) =
\begin{cases}
 0\qquad &t \in \big(0, (j-1)\frac {t^*} m\big]
\\
 \frac m{t^*}\big[t-(j-1)\frac{t^*} m\big] &t \in\big((j-1)\frac{t^*} m,j\frac{t^*} m\big)
\\
 1 &t\in\big[j\frac{t^*} m, {t^*}\big]
\end{cases}
\]
where $0\le j \le m$, $m\in\N_0$.  Note that the weak derivative $\chi_j'$ of $\chi_j$ satisfies
\[
0\le \chi_j' \le \frac m{t^*} \chi_{j-1}.
\]
Let us further define
\[
\sigma(t) = \vn{A}^4_{\infty,f^{-1}(B_{\frac34})}
\quad\text{ and }\quad
E_j(t) = \int_\Sigma |\nabla_{(2j)}A|^2\gamma^{4j+4}d\mu.
\]
Then Lemma \ref{T:c2} implies
\[
E_j'(t) + E_{j+1}(t) \le c\sigma(t)E_j(t) + c(1+\sigma(t))\varepsilon.
\]
Cutting off $E_j(t)$ in time by $\chi_j(t)$, we have for $e_j(t) = \chi_j(t)E_j(t)$
\[
e_j'(t) + \chi_j(t)E_{j+1}(t)
 \le c\sigma(t)e_j(t) + c(1+\sigma(t))\varepsilon + \frac m{t^*} \chi_{j-1}(t)E_j(t).
\]
For $t\in(0,t^*)$ integrating the above over $(0,t)$ gives
\begin{align*}
e_j(t) &+ \int_0^t\chi_j(\tau)E_{j+1}(\tau)\,d\tau
\\
 &\le   c\int_0^t \sigma(\tau)e_j(\tau)\,d\tau
     + c\,\varepsilon\int_0^t(1+\sigma(\tau))\,d\tau
     + \frac m{t^*} \int_0^t\chi_{j-1}(\tau)E_j(\tau)\,d\tau
\\
 &\le
   c\,\varepsilon
 + c\int_0^t \sigma(\tau)e_j(\tau)\,d\tau
 + \frac m{t^*} \int_0^t\chi_{j-1}(\tau)E_j(\tau)\,d\tau,
\end{align*}
where we used \eqref{EQie3}.  Using Gronwall's inequality on the above and again noting \eqref{EQie3} yields
\begin{equation}
e_j(t) + \int_0^t\chi_j(\tau)E_{j+1}(\tau)\,d\tau
\le
   c\,\varepsilon
 + ce^{c\,\varepsilon}\frac m{t^*} \int_0^t\Big(\int_0^s\chi_{j-1}(\tau)E_j(\tau)\,d\tau\Big)\sigma(s)\,ds.
\label{EQie4}
\end{equation}
For the purposes of induction, let us assume
\begin{equation}
\label{EQie5}
e_{j-1}(t) + \int_0^t\chi_{j-1}(\tau)E_{j}(\tau)\,d\tau
\le
   \frac{c\,\varepsilon}{(t^*)^{j-1}}.
\end{equation}
Then \eqref{EQie4} implies
\begin{align*}
e_j(t) + \int_0^t\chi_j(\tau)E_{j+1}(\tau)\,d\tau
&\le
   c\,\varepsilon
 + c\,\varepsilon\frac m{t^*}\frac m{(t^*)^{j-1}} \int_0^t\sigma(s)ds
\\
&\le
   c\,\varepsilon
 + c\,\varepsilon\frac m{(t^*)^{j}}
\\
&\le
  c\,\varepsilon\frac m{(t^*)^{j}},
\label{EQie4}
\end{align*}
since $t^* < \delta \le c$ by assumption.
Noting that $e_0(t) = \vn{A\gamma^2}_2^2 \le c\,\varepsilon$ and that
\[
e_0(t) + \int_0^t\chi_0(\tau)E_{1}(\tau)\,d\tau \le c\,\varepsilon
\]
by \eqref{EQie3}, we have in fact proven \eqref{EQie5} for all $1\le j \le m+1$.

The first consequence is that
\[
\int_\Sigma |\nabla_{(2m)}A|^2\gamma^{4m+4}d\mu \le \frac{c\,\varepsilon}{(t^*)^m},
\]
which is the $L^2$ estimate for even order derivatives of $A$.  For odd orders, we note that \eqref{E:int2}
implies
\[
\int_\Sigma |\nabla_{(2m+1)}A|^2\gamma^{4m+6}d\mu
 \le c\,\varepsilon\Big(
                    \int_\Sigma |\nabla_{(2m)}A|^2\gamma^{4m+4}d\mu
                  + \int_\Sigma |\nabla_{(2m+2)}A|^2\gamma^{4m+8}d\mu
                 \Big),
\]
and so the $L^2$ estimate for odd order derivatives of $A$ follows.  The $L^\infty$ estimate is obtained via \eqref{E:T}
and \eqref{E:T2}.  First apply \eqref{E:T2} to bound $\vn{A}^4_\infty$ pointwise in time, and then \eqref{E:T} to
estimate
\[
\vn{\nabla_{(k)}A}_{\infty,[\gamma=1]}^4
 \le 
\vn{\nabla_{(k)}A}_{2,[\gamma>0]}^2
\big(\vn{\nabla_{(k+2)}A}_{2,[\gamma>0]}^2
+
\vn{\nabla_{(k)}A}_{2,[\gamma>0]}^2\big).
\]
Given the $L^2$ estimates this then implies the $L^\infty$ bounds, and so we are done.
\end{proof}%}}}

%}}}

\section{Blowup analysis and asymptotic behaviour}%{{{

%Our goal is to demonstrate the development of a finite time curvature singularity for flows with initial data satisfying
%\eqref{Eftsingass}.
%Our strategy is as follows.

A priori, although the energy $\SW_{\lambda_1,\lambda_2}(f_t)$ is monotonically decreasing, we can not use this to conclude
that the Willmore energy or the surface area remain uniformly bounded by $\SW_{\lambda_1,\lambda_2}(f_0)$ along the
flow.
The flow \eqref{CW} is fourth order and highly non-linear: the surface may develop self-intersections and (assuming we
have a well-defined notion of signed volume, such as the pull-back of the Euclidean volume form by $f$) this could drive
the volume term to negative values.
In order to prevent this from occuring we first show that preservation of embeddedness holds for the flow \eqref{CW}
with initial data satisfying \eqref{Eftsingass}.
The idea behind proving this is to show that at small energy levels a conservation law for the Willmore energy holds
along the flow (cf. \cite{W10sd}).
This implies 
\[
\frac14\int_\Sigma H^2 d\mu < 8\pi,
\]
and so applying \cite[Theorem 6]{LY82ca} we obtain that $f_t$ is an embedding for every $t\in[0,T)$.
Using this we are able to directly estimate the Euler-Lagrange operator from below in $L^2$, which by an energy
dissipation argument proves $T<\infty$.

We then examine the shape of the singularity.
Due to Theorem \ref{Tlt}, we know that curvature has concentrated around some point at final time.
We use this to construct a blowup, relying on Theorem \ref{Tie} and the compactness theorem from \cite{KS01} to ensure
its existence and smoothness.
Examining this blowup we determine that (in contrast with \cite{KS01,W10sd}) it is a smooth round sphere.
This is proved by showing that the blowup is an embedded Willmore surface with non-zero curvature.
The argument does not depend on the choice of blowup sequence: for any sequence of radii we obtain a smooth round sphere.
This implies that the flow itself is asymptotic to a self-similarly shrinking round sphere, and so $f_t$ approaches a
round point.

\begin{prop}
\label{PMonotoneCurvature}
Let $f:\Sigma\times[0,T)\rightarrow\R^3$ be a constrained Willmore flow satisfying \eqref{Eftsingass} with $\lambda_1 > 0$,
$\lambda_2 \ge 0$. 
Then for any $t\in[0,T)$ we have
\[
\rD{}{t}\int_M|A^o|^2d\mu = \frac{1}{2}\rD{}{t}\int_M|H|^2d\mu
                          \le -\frac{1}{2}\int_M|\BW_{0,0}(f)|^2d\mu 
\]
and $f:\Sigma\times[0,T)\rightarrow\R^3$ is a family of embeddings.
\end{prop}
\begin{proof}%{{{
From the definition of the flow (see Lemma \ref{LMeveq}), we have that
\[
\SW_{\lambda_1,\lambda_2}(f) \le \SW_{\lambda_1,\lambda_2}(f_0) \le 4\pi + \varepsilon_2.
\]
Let us introduce the notation $\text{Vol $\Sigma_t$} = \big(\text{Vol $\Sigma$}\big)\big|_{t}$.
Since $\text{Vol $\Sigma_0$} > 0$, there is by short time existence a $\delta>0$ such that
on the half-open interval $[0,\delta)$ we have $\text{Vol }\Sigma_t > 0$.
Let us assume that $\delta$ is the largest such time with this property, i.e $\text{Vol }\Sigma_\delta = 0$.
In particular we have 
\begin{equation}
\vn{A^o}_2^2 \le 2\varepsilon_2,\quad\text{and}\quad \lambda_1\mu(\Sigma) \le \varepsilon_2\quad \text{on}\ \ [0,\delta).
\label{EQsmallness}
\end{equation}
By the Michael-Simon Sobolev inequality we estimate
\begin{equation}
\frac12\int_\Sigma |A^o|^4d\mu
 \le 4\Cs^2\vn{A^o}_2^2\int_\Sigma |\nabla A^o|^2 d\mu
   + \Cs^2\mu(\Sigma)\int_\Sigma H^2|A^o|^4 d\mu,
\label{EQbuAo4L2}
\end{equation}
\begin{align}
\frac12\int_\Sigma |A^o|^2d\mu
 &\le \Cs^2\mu(\Sigma)\int_\Sigma |\nabla A^o|^2 d\mu
    + \Cs^2\vn{H}_1\int_\Sigma |H|\,|A^o|^2 d\mu
\notag
\\
 &\le \Cs^2\mu(\Sigma)\int_\Sigma |\nabla A^o|^2 d\mu
    + \Cs^2\mu(\Sigma)^\frac12\vn{H}_2\int_\Sigma |H|\,|A^o|^2 d\mu,
\label{EQbuAoL2}
\end{align}
and
\begin{equation}
\frac12\int_\Sigma |\nabla A^o|^2d\mu
 \le \Cs^2\mu(\Sigma)\int_\Sigma (|\nabla_{(2)} A^o|^2 + H^2|\nabla A^o|^2)\, d\mu,
\label{EQbugradAoL2}
\end{equation}
where $\Cs = {4^3}/\sqrt{\pi}$.
Note that by the Gauss-Bonnet theorem and short time existence we have for each $t\in[0,T)$
\[
\frac14\rD{}{t}\int_\Sigma H^2d\mu
=
  \rD{}{t}\int_\Sigma K\,d\mu
  +
  \frac12\rD{}{t}\int_\Sigma |A^o|^2d\mu
=
  \frac12\rD{}{t}\int_\Sigma |A^o|^2d\mu.
\]
We now compute
\begin{align}
\frac12\rD{}{t} \int_\Sigma |A^o|^2d\mu
  &= -\frac12 \int_\Sigma (\Delta H + H|A^o|^2)(\Delta H + H|A^o|^2 - 2H\lambda_1 - 2\lambda_2) d\mu
\notag
\\
  &= -\frac12 \int_\Sigma |\BW_{0,0}(f)|^2 d\mu
    + \lambda_1 \int_\Sigma H(\Delta H + H|A^o|^2) d\mu
\notag
\\ &\quad
  + \lambda_2 \int_\Sigma (\Delta H + H|A^o|^2) d\mu
\notag
\\
  &= -\frac12 \int_\Sigma |\BW_{0,0}(f)|^2 d\mu
    + \lambda_1 \int_\Sigma (-|\nabla H|^2 + H^2|A^o|^2) d\mu
\notag
\\ &\quad
    + \lambda_2 \int_\Sigma H|A^o|^2 d\mu
\notag
\\
  &= -\frac12 \int_\Sigma |\BW_{0,0}(f)|^2 d\mu
    - 2\lambda_1 \int_\Sigma |\nabla A^o|^2 d\mu
    + 2\lambda_1 \int_\Sigma |A^o|^4 d\mu
\notag
\\ &\quad
    + \lambda_2 \int_\Sigma H|A^o|^2 d\mu,
\label{EQbuevoAoL2}
\end{align}
where we used the evolution equations Lemma \ref{LMappevo}, the divergence theorem, and the identity
\begin{align}
\int_\Sigma &|\nabla A^o|^2d\mu
    + \frac12\int_\Sigma H^2|A^o|^2d\mu
  = \frac12\int_\Sigma |\nabla H|^2d\mu
    + \int_\Sigma |A^o|^4d\mu.
\label{EQintgradHgradAoapp}
\end{align}
We have included a proof of \eqref{EQintgradHgradAoapp} for the reader's convenience in the Appendix.

Let us first assume $\lambda_2=0$.
There exists a constant $c_3$ such that 
\begin{equation}
\int_\Sigma \big(|\nabla_{(2)}A|^2
                + |A|^2|\nabla A|^2
                + |A|^4|A^o|^2\big)\,d\mu
\le c_3\int_\Sigma |\BW_{0,0}(f)|^2d\mu
\label{EQbuks1}
\end{equation}
holds.
Estimate \eqref{EQbuks1} follows by taking $\rho\rightarrow\infty$ (note that $\Sigma$ is closed) in \cite[Proposition
2.6]{KS01}.
Combining \eqref{EQbuAo4L2} and \eqref{EQbuks1} with the simple estimate
\[
4\lambda_1\Cs^2 
     \frac{\varepsilon_2}{\lambda_1}\int_\Sigma H^2|A^o|^4 d\mu
\le
8\lambda_1\Cs^2 
     \frac{\varepsilon_2}{\lambda_1}\int_\Sigma |A|^4|A^o|^2 d\mu
\]
we compute
\begin{align*}
\frac12\rD{}{t} \int_\Sigma |A^o|^2d\mu
  &= -\frac12 \int_\Sigma |\BW_{0,0}(f)|^2 d\mu
    - 2\lambda_1 \int_\Sigma |\nabla A^o|^2 d\mu
    + 2\lambda_1 \int_\Sigma |A^o|^4 d\mu
\\
  &\le -\frac14 \int_\Sigma |\BW_{0,0}(f)|^2 d\mu
       -\frac1{4c_3}
        \int_\Sigma \big(|\nabla_{(2)}A|^2
                + |A|^2|\nabla A|^2
                + |A|^4|A^o|^2\big)d\mu
\\*
&\quad
    - 2\lambda_1 \int_\Sigma |\nabla A^o|^2 d\mu
    + 4\lambda_1\Cs^2 \Big(
     8\varepsilon_2\int_\Sigma |\nabla A^o|^2 d\mu + \frac{\varepsilon_2}{\lambda_1}\int_\Sigma H^2|A^o|^4 d\mu
                      \Big)
\\
  &\le -\frac14 \int_\Sigma |\BW_{0,0}(f)|^2 d\mu
\\*
&\quad
       -\Big(
             \frac1{4c_3}-8\varepsilon_2\Cs^2
        \Big)
        \int_\Sigma \big(|\nabla_{(2)}A|^2
                + |A|^2|\nabla A|^2
                + |A|^4|A^o|^2\big)d\mu
\\*
&\quad
    - 2\lambda_1\Big(
	  1 - 16\varepsilon_2\Cs^2
      \Big) \int_\Sigma |\nabla A^o|^2 d\mu.
\end{align*}
The result follows for $t\in[0,\delta)$ so long as $\varepsilon_2 \le \frac1{32\Cs^2}\min\{1, (2c_3)^{-1}\}$.

Let us now assume $\lambda_2>0$.
Young's inequality and the estimate \eqref{EQbuAoL2} above imply
\begin{align}
\int_\Sigma |H|\,|A^o|^2d\mu
 &\le
\delta\int_\Sigma H^4|A^o|^2d\mu
+
\frac3{4^{\frac43}\delta^{\frac13}}\int_\Sigma |A^o|^2d\mu
\notag
\\
 &\le
      \delta\int_\Sigma H^4|A^o|^2d\mu
    + \frac6{4^{\frac43}\delta^{\frac13}}\Cs^2
      \frac{\varepsilon_2}{\lambda_1}\int_\Sigma |\nabla A^o|^2 d\mu
\notag
\\
&\quad
         + \frac6{4^{\frac43}\delta^{\frac13}}\Cs^2\frac{\sqrt\varepsilon_2}{\sqrt\lambda_1}\sqrt{16\pi+4\varepsilon_2}
           \int_\Sigma |H|\,|A^o|^2 d\mu.
\label{EQbuHAo2L1}
\end{align}
Choose $\delta = \frac{1}{32\lambda_2c_3}$ and further assume that $\varepsilon_2$ satisfies
\[
\varepsilon_2(4\pi+\varepsilon_2)
 \le \lambda_1\bigg(\frac{4^\frac13\delta^\frac13}{6\, C_S^2}\bigg)^2.
\]
This implies $1-\frac{6\Cs^2\sqrt\varepsilon_2}{4^{\frac43}\delta^{\frac13}\sqrt\lambda_1}\sqrt{16\pi+4\varepsilon_2} \ge
\frac12$, and so \eqref{EQbuHAo2L1} with these choices gives
\begin{align}
\frac12\int_\Sigma |H|\,|A^o|^2d\mu
 &\le \Big(1-\frac{6\Cs^2\sqrt\varepsilon_2}{4^{\frac43}\delta^{\frac13}\sqrt\lambda_1}\sqrt{16\pi+4\varepsilon_2}\Big)
      \int_\Sigma |H|\,|A^o|^2d\mu
\notag
\\
 &\le
      \frac1{32\lambda_2c_3}\int_\Sigma H^4|A^o|^2d\mu
    + \frac{\sqrt\varepsilon_2}{4\sqrt\lambda_1\sqrt{4\pi+\varepsilon_2}} \int_\Sigma |\nabla A^o|^2 d\mu
\notag
\\
 &\le
      \frac1{8\lambda_2c_3}\int_\Sigma |A|^4|A^o|^2d\mu
\notag
\\
&\quad
    + \varepsilon_2\frac{\Cs^2\sqrt\varepsilon_2}{2\lambda_1^\frac32\sqrt{4\pi+\varepsilon_2}}
 \Big(\int_\Sigma |\nabla_{(2)} A^o|^2 + H^2|\nabla A^o|^2\, d\mu\Big),
\label{EQbulambda2est}
\end{align}
where in the last step we applied \eqref{EQbugradAoL2}.
Using estimates \eqref{EQbuAo4L2}, \eqref{EQbuks1} and \eqref{EQbulambda2est} we now compute
\begin{align*}
\frac12\rD{}{t} \int_\Sigma |A^o|^2d\mu
  &= -\frac12 \int_\Sigma |\BW_{0,0}(f)|^2 d\mu
    - 2\lambda_1 \int_\Sigma |\nabla A^o|^2 d\mu
    + 2\lambda_1 \int_\Sigma |A^o|^4 d\mu
\\
&\quad
    + \lambda_2 \int_\Sigma H|A^o|^2 d\mu.
\\
  &\le -\frac14 \int_\Sigma |\BW_{0,0}(f)|^2 d\mu
    - 2\lambda_1\Big(
	  1 - 16\varepsilon_2\Cs^2
      \Big)\int_\Sigma |\nabla A^o|^2 d\mu
\\
&\quad
        -\Big(
             \frac1{4c_3}-8\varepsilon_2\Cs^2
        \Big)
        \int_\Sigma \big(|\nabla_{(2)}A|^2
                       + |A|^2|\nabla A|^2
                       + |A|^4|A^o|^2\big)d\mu
\\
&\quad
    + \frac1{4c_3}\int_\Sigma |A|^4|A^o|^2d\mu
    + \varepsilon_2\frac{\lambda_2\Cs^2\sqrt\varepsilon_2}{\lambda_1^\frac32\sqrt{4\pi+\varepsilon_2}}
 \Big(\int_\Sigma |\nabla_{(2)} A^o|^2 + H^2|\nabla A^o|^2\, d\mu\Big)
\\
  &\le -\frac14 \int_\Sigma |\BW_{0,0}(f)|^2 d\mu
\\
&\quad
       -\Big(
             \frac1{4c_3}-8\varepsilon_2\Cs^2
           - \varepsilon_2\frac{2\lambda_2\Cs^2\sqrt\varepsilon_2}{\lambda_1^\frac32\sqrt{4\pi+\varepsilon_2}}
        \Big)
        \int_\Sigma \big(|\nabla_{(2)}A|^2
                       + |A|^2|\nabla A|^2
                       \big)d\mu
\\
&\quad
    - 2\lambda_1\Big(
	  1 - 16\varepsilon_2\Cs^2
      \Big)\int_\Sigma |\nabla A^o|^2 d\mu.
\end{align*}
The result follows for $t\in[0,\delta)$ so long as $\varepsilon_2 \le \frac1{2^6\Cs^2}\min\{1,
 4\lambda_1^\frac32(\lambda_2c_3)^{-1},
 (2c_3)^{-1}
\}$.

In each case we have shown that for every $t\in[0,\delta)$
\begin{equation}
\rD{}{t}\int_M|A^o|^2d\mu \le -\frac{1}{2}\int_M|\BW_{0,0}(f)|^2d\mu
\quad\text{and}\quad 
\frac{1}{4}\int_M|H|^2d\mu \le 4\pi+\varepsilon_2 < 8\pi.
\label{EQmonoton}
\end{equation}
It then follows from Theorem 6 in \cite{LY82ca} that each $f_t$ is an embedding.
There are two possibilities: either $\delta=T$, or $\delta < T$.
In the former case we are already finished, since then equation \eqref{EQmonoton} holds for every $t\in[0,T)$ as
desired.
In the latter case we have by short time existence a smooth non-singular surface $f_\delta$ (since otherwise
$\delta=T$).
Again by short time existence we have that $\vn{H}_2^2$ is a lower semicontinuous function of time, and so by
\eqref{EQmonoton} above $\vn{H}_2^2 < 32\pi$ at $t=\delta$.
This implies that $f_\delta$ is embedded and $\text{Vol }\Sigma_\delta > 0$.
This is however a contradiction with the maximality of $\delta$, and so we are finished.
\end{proof}%}}}

We now require the following estimate.

\begin{prop}
Let $f:\Sigma\times[0,T)\rightarrow\R^3$ be a constrained Willmore flow satisfying \eqref{Eftsingass} with  $\lambda_1
\ge 0, \lambda_2 \ge 0$.
Then for any $t\in[0,T)$ we have
\begin{equation}
\label{EQbuAo4infty}
\int_0^t \vn{A^o}_{\infty}^4 d\tau
< c_4\,\varepsilon_2^2\,(1+t),
\end{equation}
where $c_4$ is a constant depending on $\lambda_1$, $\lambda_2$, $\SW_{\lambda_1,\lambda_2}(f_0)$ only.
\label{PNa04linftyest}
\end{prop}
\begin{proof}%{{{
Proposition \ref{PMonotoneCurvature} implies that $f$ is a family of embeddings, and so $\text{Vol }\Sigma \ge 0$.
This implies in particular that \eqref{EQsmallness} holds with $\delta=T$.
Let us estimate
\[
4\lambda_1\vn{A^o}_4^4 \le 2\lambda_1^2\vn{A^o}_2^2 + 2 \vn{A^o}_6^6
\]
and, using again Proposition \ref{PMonotoneCurvature},
\[
\lambda_2\int_\Sigma H|A^o|^2d\mu \le 
4\lambda_2\vn{A^o}_4^4 + \frac{\lambda_2}{16} \int_\Sigma |H|^2d\mu
 \le 2\lambda_2^2\vn{A^o}_2^2 + 2\vn{A^o}_6^6 + 
    \frac{\lambda_2}{4}\SW_{\lambda_1,\lambda_2}(f_0).
\]
Combining these estimates with \eqref{EQbuevoAoL2} and keeping in mind that $\lambda_1\ge0$ it follows that
\begin{align}
\int_0^t \int_\Sigma |\BW_{0,0}(f)|^2 d\mu\, d\tau
  &\le 
    2\varepsilon_2
    + 4\int_0^t\int_\Sigma |A^o|^6 d\mu\, d\tau
    + t\varepsilon_2\Big( 4\lambda_1^2+4\lambda_2^2 + \frac{\lambda_2}{4}\Big).
\label{EQL1inTforEL}
\end{align}
A straightforward combination of \cite[Lemma 2.5]{KS01} and \cite[Proposition 2.6]{KS01} and taking
$\rho\rightarrow\infty$ (recall $\Sigma$ is closed) yields
\[
\int_\Sigma |A^o|^6 d\mu \le c\varepsilon_2 \int_\Sigma |\BW_{0,0}(f)|^2 d\mu. 
\]
Using this to estimate the right hand side of \eqref{EQL1inTforEL} we obtain
\begin{align*}
\int_0^t \int_\Sigma |\BW_{0,0}(f)|^2 d\mu\, d\tau
  &\le 
    2\varepsilon_2
    + 4c\varepsilon_2\int_0^t \int_\Sigma |\BW_{0,0}(f)|^2 d\mu\, d\tau
    + t\varepsilon_2\Big( 4\lambda_1^2+4\lambda_2^2 + \frac{\lambda_2}{4}\Big),
\end{align*}
where $c$ is an absolute constant.
Absorbing on the left we find
\begin{align*}
\int_0^t \int_\Sigma |\BW_{0,0}(f)|^2 d\mu\, d\tau
  &\le 
    c\,\varepsilon_2\,(1+t),
\end{align*}
where $c$ is a constant depending on $T$, $\lambda_1$, $\lambda_2$, and $\SW_{\lambda_1,\lambda_2}(f_0)$ only.
The flow independent estimate \cite[Theorem 2.9]{KS01} with $\gamma\equiv1$ implies that there exists an absolute
constant $c$ such that
\[
\vn{A^o}_\infty^4
 \le 
  c\vn{A^o}_2^2\int_\Sigma |\BW_{0,0}(f)|^2 d\mu,
\]
which, when combined with the estimate above, yields \eqref{EQbuAo4infty}.
\end{proof}%}}}

\begin{prop}
Suppose $f:\Sigma\times[0,T)\rightarrow\R^3$ is a constrained Willmore flow with $\lambda_1 > 0$,
$\lambda_2 \ge 0$ satisfying \eqref{Eftsingass}.
Then $T < c_5 < \infty$, where
\[
c_5 =
  \frac1{4\lambda_1^2\pi}
\SW_{\lambda_1,\lambda_2}(f_0) + 1.
\]
%with   
%\[
%c_6 = \lambda_1^2 - \lambda_2\lambda_1^\frac12\varepsilon_2^\frac12,\quad\quad 
%c_7 = \frac{\lambda_2^2}{\lambda_1^3}+4,
%\]
%and $c_4$ as in Proposition \ref{PNa04linftyest}.
\label{PNfinite}
\end{prop}
\begin{proof}%{{{
Proposition \ref{PMonotoneCurvature} implies that $f$ is a family of embeddings, and so \eqref{EQsmallness} holds with
$\delta=T$.
By the definition of the flow (see Lemma \ref{LMeveq}), we have
\begin{align}
\rD{}{t}\SW_{\lambda_1,\lambda_2}(f)
 &= -\frac{1}{2}\int_\Sigma |\Delta H + H|A^o|^2-2\lambda_1H-2\lambda_2|^2 d\mu
\notag
\\
 &= -\frac{1}{2}\int_\Sigma |\BW_{0,0}(f)|^2 d\mu
    - 2\int_\Sigma |\lambda_1H+\lambda_2|^2 d\mu
\notag\\&\quad
    + 2\int_\Sigma \big(\Delta H + H|A^o|^2\big)\big(\lambda_1H+\lambda_2\big)\, d\mu.
\label{EQbuAoevoL2}
\end{align}
Using \eqref{EQintgradHgradAoapp} we rewrite the last two terms as
\begin{align}
- 2\int_\Sigma |\lambda_1H+\lambda_2|^2 d\mu
&+ 2\int_\Sigma \big(\Delta H + H|A^o|^2\big)\big(\lambda_1H+\lambda_2\big)\, d\mu
\notag\\
&= - 2\lambda_1^2\int_\Sigma H^2 d\mu - 4\lambda_1\lambda_2\int_\Sigma H\, d\mu - 2\lambda_2^2|\Sigma|
\notag\\&\quad
     - 2\lambda_1 \int_\Sigma \big(|\nabla H|^2 - H^2|A^o|^2\big)\,d\mu
     + 2\lambda_2\int_\Sigma H|A^o|^2d\mu
\notag\\
&= - 2\lambda_1^2\int_\Sigma H^2 d\mu - 4\lambda_1\lambda_2\int_\Sigma H\,d\mu - 2\lambda_2^2|\Sigma|
\notag\\
&\qquad
     - 4\lambda_1 \int_\Sigma |\nabla A^o|^2 d\mu + 4\lambda_1\int_\Sigma |A^o|^4 d\mu
     + 2\lambda_2\int_\Sigma H|A^o|^2d\mu.
\label{EQbuextras}
\end{align}
Noting that for closed surfaces $4\sqrt\pi \le \vn{H}_2 < \vn{H}_2^2$, we estimate
\begin{equation}
-4\lambda_1\lambda_2 \int_\Sigma H\, d\mu 
\le 4\lambda_1\lambda_2 \sqrt{\mu(\Sigma)}\Big(\int_\Sigma H^2 d\mu\Big)^{1/2}
< 4\lambda_2\sqrt{\varepsilon_2\lambda_1} \int_\Sigma H^2 d\mu.
\label{EQbuTfinite1}
\end{equation}
We shall also use (recall $\lambda_1>0$)
\begin{equation}
2\lambda_2\int_\Sigma H|A^o|^2d\mu
\le \lambda_1^2\int_\Sigma H^2 d\mu
  + \frac{\lambda_2^2}{\lambda_1^2}\int_\Sigma |A^o|^4d\mu.
\label{EQbuTfinite2}
\end{equation}
Combining \eqref{EQbuTfinite1} and \eqref{EQbuTfinite2} with \eqref{EQbuextras} we have
\begin{align}
- 2\int_\Sigma |\lambda_1H+\lambda_2|^2 d\mu
&+ 2\int_\Sigma \big(\Delta H + H|A^o|^2\big)\big(\lambda_1H+\lambda_2\big)\, d\mu
\notag\\
&< - 4\lambda_1\int_\Sigma |\nabla A^o|^2 d\mu
     - 2\lambda_2^2\mu(\Sigma)
     - \big(\lambda_1^2
        - \lambda_2\lambda_1^\frac12\varepsilon_2^\frac12\big)\int_\Sigma H^2 d\mu
\notag\\
&\quad
     + \Big(\frac{\lambda_2^2}{\lambda_1^2}+4\lambda_1\Big)\vn{A^o}_4^4
\notag\\
&\le
     - \big(\lambda_1^2
        - \lambda_2\lambda_1^\frac12\varepsilon_2^\frac12\big)\int_\Sigma H^2 d\mu
     + \Big(\frac{\lambda_2^2}{\lambda_1^2}+4\lambda_1\Big)\mu(\Sigma)\,\vn{A^o}_\infty^4
\notag\\
&\le
     - 16c_6\pi
     + c_7\,\varepsilon_2\vn{A^o}_\infty^4,
\label{EQbuTfinite3}
\end{align}
where 
\[
c_6 = \frac{\lambda_1^2}{2},\quad\text{and}\quad 
c_7 = \frac{\lambda_2^2}{\lambda_1^3}+4.
\]
Here we need $\varepsilon_2 \le \frac{\lambda_1^3}{4\lambda_2^2}$ so that
\[
\lambda_1^2 - \lambda_2\lambda_1^\frac12\varepsilon_2^\frac12 \ge \frac{\lambda_1^2}{2}.
\]
Integrating \eqref{EQbuAoevoL2} and estimating the right hand side with \eqref{EQbuAo4infty}, \eqref{EQbuTfinite3}, we obtain
\begin{align}
\SW_{\lambda_1,\lambda_2}(f_t)
&<
      \SW_{\lambda_1,\lambda_2}(f_0)
    + c_4\,c_7\,\varepsilon_2^3(1+t)
    - (16c_6\pi)t 
\notag\\
&<
      \SW_{\lambda_1,\lambda_2}(f_0)
    + c_4\,c_7\,\varepsilon_2^3
    - t\big(
        16c_6\pi
      - c_4\,c_7\,\varepsilon_2^3
       \big)
\notag\\
&<
      \SW_{\lambda_1,\lambda_2}(f_0)
    + 8c_6\pi\,(1-t)
\label{EQbuTfinite4}
\end{align}
assuming $\varepsilon_2 \le 2\sqrt[3]{\frac{c_6\pi}{c_4c_7}}$.
Since $\text{Vol }\Sigma$, $\lambda_1$, $\lambda_2 \ge 0$ we have $\SW_{\lambda_1,\lambda_2}(f) \ge 0$.
This implies that
\[
T\le
  \frac1{8c_6\pi}
\big(\SW_{\lambda_1,\lambda_2}(f_0) + 8c_6\pi\big),
\]
since otherwise there would exist a $t^*\in[0,T)$ such that
$t^*\ge
  \frac1{8c_6\pi}
\big(\SW_{\lambda_1,\lambda_2}(f_0) + 8c_6\pi\big)$,
which is in contradiction with \eqref{EQbuTfinite4}.
\end{proof}%}}}

\begin{proof}[Proof of Theorem \ref{Tftsing}]%{{{
Proposition \ref{PNfinite} implies that $T<\infty$; it remains to classify the asymptotic shape of the singular surface
$f_T(\cdot)$. We know by Theorem \ref{Tlt} that for any sequence of radii $r_j\searrow 0$ there exists a sequence of
times $t_j\nearrow T$ such that
\begin{equation*}
t_j = \inf\Big\{t\ge0: \sup_{x\in\R^3} \int_{f^{-1}(B_{r_j}(x))}|A|^2d\mu
 > \varepsilon_3\Big\} < T,
\end{equation*}
where $\varepsilon_3 = \varepsilon_0/c_0$ and $\varepsilon_0, c_0$ are as in the Theorem \ref{Tlt}.
Arguing as in \cite{KS01}, we know that
\[
\int_{f^{-1}(B_{r_j}(x))}|A|^2d\mu\bigg|_{t=t_j} \le \varepsilon_3
\text{  for any }x\in\R^3,
\]
and
\begin{equation}
\int_{f^{-1}\overline{(B_{r_j}(x_j))}}|A|^2d\mu\bigg|_{t=t_j} \ge \varepsilon_3
\text{  for some }x_j\in\R^3.
\label{EQbusomecurv}
\end{equation}
Consider the rescaled immersions
\[
f_j:\Sigma\times\big[-r_j^{-4}t_j, r_j^{-4}(T-t_j)\big)\rightarrow\R^3,
\qquad
f_j(p,t) = \frac{1}{r_j}\big(f(p,t_j+r_j^4t)-x_j\big).
\]
Theorem \ref{Tlt} implies $r_j^{-4}(T-t_j) \ge c_0$ for any $j$ and also that
\[
\sup_{x\in\R^3}\int_{f_j^{-1}(B_{1}(x))}|A|^2d\mu \le \varepsilon_0
\text{  for }0<t\le c_0.
\]
Using Theorem \ref{Tie} on parabolic cylinders $B_1(x)\times(t-1,t]$ as in
\cite{KS01} we obtain
\[
\vn{\nabla_{(k)}A}_{\infty,f_j} \le c(k)\quad\text{for}\quad-r_j^{-4}t_j+1\le t\le c_0.
\]
The Willmore energy is bounded and so a local area bound may be obtained as in \cite{KS01}
from a lemma due to Simon \cite{simon1993existence}.  Therefore applying Theorem 4.2 from
\cite{KS01} to the sequence $f_j = f_j(\cdot,0):\Sigma\rightarrow\R^3$ we recover a limit
immersion $\hat{f}_0:\hat{\Sigma}\rightarrow\R^3$, where $\hat{\Sigma} \cong \Sigma$.  We also obtain the
diffeomorphisms $\phi_j:\hat{\Sigma}(j)\rightarrow U_j\subset \Sigma$.
The reparametrisation
\begin{equation*}
  f_j(\phi_j,\cdot):\hat{\Sigma}(j)\times[0,c_0]\rightarrow\R^3
\end{equation*}
is a locally constrained Willmore flow with initial data
\begin{equation*}
f_j(\phi_j,0) = \hat{f}_0+u_j:\hat{\Sigma}(j)\rightarrow\R^3.
\end{equation*}
Arguing again as in \cite{KS01} we obtain the locally smooth convergence
\begin{equation}
f_j(\phi_j,\cdot) \rightarrow \hat{f},
\label{C7E5}
\end{equation}
where $\hat{f}:\hat{\Sigma}\times[0,c_0]\rightarrow\R^3$ is a locally constrained Willmore flow with initial data
$\hat{f}_0$.
We wish to show that this blowup is a critical point for the Willmore functional.

\begin{thm}
Let $f:\Sigma\times[0,T)\rightarrow\R^3$ be a constrained Willmore flow with $\lambda_1 > 0$, $\lambda_2
\ge 0$ satisfying \eqref{Eftsingass}.
Then the blowup $\tilde{f}$ as constructed above is an embedded Willmore surface.
\label{Tbu}
\end{thm}
\begin{proof}%{{{
Noting the scale invariance of $\vn{A^o}^2_2$, we use Proposition \ref{PMonotoneCurvature} to
compute
\begin{align*}
2\int_0^{c_0}\int_{\tilde{\Sigma}(j)}&|\BW_{0,0}(f_j(\phi_j,t))|^2d\mu_{f_j(\phi_j,\cdot)}dt
= 2\int_0^{c_0}\int_{U_j}|(\BW_{0,0})_j|^2d\mu_jdt
\\
&\le
   \int_\Sigma |A^o_j(0)|^2d\mu_j
 - \int_\Sigma |A^o_j(c_0)|^2d\mu_j
\\
&= 
   \int_\Sigma |A^o(t_j)|^2d\mu
 - \int_\Sigma |A^o(t_j+r_j^4c_0)|^2d\mu,
\end{align*}
and this converges to zero as $j\rightarrow\infty$.  Therefore $\BW_{0,0}(\tilde{f}) \equiv 0$ and
the blowup $\tilde{f}$ is an embedded (by \cite[Theorem 6]{LY82ca} and Proposition \ref{PMonotoneCurvature}) Willmore
surface.
\end{proof}%}}}

Using again the scale invariance of $\vn{A^o}_2^2$, \cite[Theorem 2.7]{KS01} thus implies that
$\tilde{f}$ is a union of planes and spheres.  Ruling out disconnected components using \cite[Lemma 4.3]{KS01} and
noting that by \eqref{EQbusomecurv} we have $\vn{\tilde{A}}_2^2 > 0$, we conclude that $\tilde{f}$ is a round sphere.

As the sequence of radii was arbitrary, this shows that $\mu(\Sigma) \rightarrow 0$ and that $f_t$ is asymptotic to
a round point.
\end{proof}%}}}

%}}}

\appendix
\section{Selected proofs}%{{{

We collect here the proofs of several well-known formulae and results for the convenience of the reader and readability
of the paper.
Many of the statements contained in this appendix have appeared in a similar form in
\cite{KS01,KS02,MWW10csdlt,W10,W10slt,W10sd}.
Throughout this appendix $M$ denotes a smooth an $n$-dimensional reference manifold and $f:M^n\rightarrow\R^{n+1}$ is an
immersion of $M$.  We denote by $F$ an (unless otherwise stated) arbitrary function.

\begin{lem}
For $f:M^n\times[0,T)\rightarrow\R^{n+1}$ evolving by $\pD{}{t}f = F\nu$ the following equations
hold:
\begin{align*}
  \pD{}{t}g_{ij} &= -2FA_{ij},\quad
  \pD{}{t}g^{ij}  =  2FA^{ij},\quad
  \rD{}{t}d\mu = -(HF)d\mu,\\
  \pD{}{t}\nu &= -\nabla F,\quad
  \pD{}{t}A_{ij} = \nabla_{ij}F - FA_i^pA_{pj},\\
  \pD{}{t}H &=  \Delta F + F |A|^2,\quad
  \pD{}{t}\Gamma = FP_1^1(A) + A*\nabla F,\text{ and }\\
  \pD{}{t}A^o_{ij} &=  S^o(\nabla_{(2)}F)
   - F\Big( (A^o)_{ik}(A^o)^k_j + \frac{1}{n}g_{ij}|A^o|^2 \Big),
\end{align*}
where $S^o(T)$ denotes the tracefree part of a symmetric bilinear form $T$.
\label{LMappevo}
\end{lem}
\begin{proof}%{{{
We begin by proving that the evolution of the unit normal $\nu$ is given by
\[
\pD{\nu}{t}
 = -g^{ij}\pD{F}{x^i}\pD{f}{x^j}
 = -\IP{\partial F}{\partial f}_g
 = -\nabla F.
\]
Since $\IP{\nu}{\nu} = 1$, any derivative of the normal is again normal to $\nu$, hence tangential to $M$.
Since $\{\partial_if\}$ forms a basis of $TM$, we may express the time derivative of $\nu$ as
\begin{equation}
\pD{\nu}{t} = g^{ij}\IP{\pD{\nu}{t}}{\pD{f}{x^i}}\pD{f}{x^j}.
\label{eqDnDtlincombasis}
\end{equation}
As $\IP{\nu}{\partial_if} = 0$,
\begin{align*}
0 &=  \IP{\pD{\nu}{t}}{\pD{f}{x^i}} + \IP{\nu}{\pD{}{t}\pD{f}{x^i}}\\
  &=  \IP{\pD{\nu}{t}}{\pD{f}{x^i}} + \IP{\nu}{\nu\pD{F}{x^i}}
   + F\IP{\nu}{\pD{\nu}{x^i}}\\
  &=  \IP{\pD{\nu}{t}}{\pD{f}{x^i}} + \pD{F}{x^i},
\end{align*}
so
\begin{equation}
\IP{\pD{\nu}{t}}{\pD{f}{x^i}} = -\pD{F}{x^i}.
\label{eqDkerDxInDnuDphi}
\end{equation} 
Substituting \eqref{eqDkerDxInDnuDphi} into \eqref{eqDnDtlincombasis}, we have
\begin{equation*} \pD{\nu}{t} = -g^{ij}\pD{F}{x^i}\pD{f}{x^j},\end{equation*}
as required.
We now move on to proving that the induced metric $(g_{ij})$ and its inverse $(g^{ij})$ evolve by
\begin{equation}
\label{eqEvIndMetric}
\pD{}{t}g_{ij} = -2F A_{ij},\quad\text{ and }\quad\pD{}{t}g^{ij} = 2F A^{ij},
\end{equation}
respectively.
Note first that
\begin{equation*}
\IP{\pD{}{t}\pD{f}{x^i}}{\pD{f}{x^j}}
 = \pD{F}{x^i}\IP{\nu}{\pD{f}{x^j}} + \IP{F\pD{\nu}{x^i}}{\pD{f}{x^j}}
 = F\IP{\pD{\nu}{x^i}}{\pD{f}{x^j}}
 = -F A_{ij},
\end{equation*}
where we used the definition of $A_{ij}$ in the last step.  Since the second fundamental form is
symmetric, we have
\begin{equation*}
\pD{}{t}g_{ij}
 = -FA_{ij} -FA_{ji}
 = -2F A_{ij}.
\end{equation*}
%This establishes the first claim. For the second, we first differente $g^{ik}g_{jk} = \delta^i_j$,
For the inverse, we first differentiate $g^{ik}g_{jk} = \delta^i_j$:
\begin{equation*}
0  = \pD{(g^{ik}g_{jk})}{t}
   = \pD{g^{ik}}{t}g_{jk} + g^{ik}(-2F A_{jk}).
\end{equation*}
Contraction gives
\begin{equation*} 
g^{jl}g_{jk}\pD{g^{ik}}{t}
 = \pD{g^{il}}{t} 
 = 2g^{jl}g^{ik}F A_{jk}
 = 2F A^{il},
\end{equation*}
and with the substitution $l\leftrightarrow j$, this finishes the proof of \eqref{eqEvIndMetric}.
We will next need to make use of the rule for differentiating determinants: For a matrix $(M)$ with differentiable
entries depending on $x$,
\[
\pD{}{x}\text{det }M = \text{det }M\ M^{ij} \pD{}{x}M_{ij}.
\]
To see this, let $\text{adj }M$ denote the adjoint matrix of $M$, which by definition satisfies $M_{ij}\,\text{adj
}M^{jk} = \delta_i^k\text{det }M$. Then
\[
\pD{}{x}\text{det }M
 = \Big(\pD{}{M_{ij}}\text{det }M\Big) \pD{}{x}M_{ij}
 = (\text{adj }M)\Big(\pD{}{M_{ij}}\text{det }M\Big) \pD{}{x}M_{ij}
 = \text{det }M\ M^{ij} \pD{}{x}M_{ij}.
\] 
We claim that the measure evolves according to
\begin{equation}
\label{EvMeasure}
\rD{}{t} d\mu = -HFd\mu.
\end{equation}
%The induced surface measure $d\mu$ on $M$ is given by
%\[d\mu = \sqrt{\text{det }(g_{ij})}dx.\]
Differentiating,
\[
\pD{}{t}d\mu
 = \pD{}{t}\left(\sqrt{\text{det }g}\,d\SH^3\right)
 = \frac{1}{2}\sqrt{\text{det }g}\ g^{ij}\pD{}{t}g_{ij}\,d\SH^3
 = -F g^{ij}A_{ij} d\mu
 = -HFd\mu,
\]
where we used the evolution of $g_{ij}$ in the last equality.  This shows \eqref{EvMeasure}.
We shall now consider the evolution of the components of the second fundamental form, $A_{ij}$.
We shall show
\begin{equation}
\label{EvCurv}
\pD{}{t}A_{ij} = \nabla_{ij}F - F A_{ik}A^k_j.
\end{equation}
%Recall that the components $A_{ij}$ are given by
%\[A_{ij} = -\left(\pD{}{x^i}\pD{}{x^j}f\ \Big\vert\ \nu\right).\]
%To simplify the notation, note first that
%\[\pD{\nu}{t} = g^{ij}\pD{F}{x^i}\pD{f}{x^j} = \sum_{i=1}^n dx^i \otimes \nabla_iF
%= \nabla F.\]
Differentiating,
\begin{align*}
\pD{}{t}A_{ij}
 &=  \IP{\pD{}{x^i}\pD{}{x^j}\pD{f}{t}}{\nu}
   + \IP{\pD{}{x^i}\pD{}{x^j}f}{\pD{\nu}{t}}
\\
&= \pD{}{x^i}\pD{}{x^j}F
 + F \IP{\pD{}{x^i}\pD{}{x^j}\nu}{\nu}
 - \IP{\Gamma^k_{ij}\pD{f}{x^k} - A_{ij}\nu}{\nabla F}
\\
&= \pD{}{x^i}\pD{}{x^j}F
 - F \IP{\pD{}{x^i}\Big(
                      A_{jk}g^{kl}\pD{f}{x^l}
                     \Big)}{\nu}
 - \Gamma^k_{ij}\IP{\pD{f}{x^k}}{\nabla F}
\\
&=
  \Big(\pD{}{x^i}\pD{}{x^j}F
 - \Gamma^k_{ij}\pD{F}{x^k}\Big)
 - F A_{jk}g^{kl}\IP{\pD{}{x^i}\pD{}{x^l}f}{\nu}
\\
&=  \nabla_{ij}F - F A_{jk}g^{kl}A_{li}\\
&=  \nabla_{ij}F - F A_{ik}A^k_j.
\end{align*}
In the above we used the Gauss-Weingarten relations
\[
\pD{}{x^i}\pD{}{x^j}f = \Gamma^k_{ij}\pD{f}{x^k} + A_{ij}\nu,
\ \ 
\pD{}{x^j}\nu = -A_{jk}g^{kl}\pD{f}{x^l},
\]
and the definition of the covariant derivative of a one-form.
%Among other things, we used the fact that
%\[\Gamma^k_{ij} = g^{kl}\left(\pD{}{x^i}\pD{}{x^j}f\ \Big\vert\ \pD{}{x^l}f\right).\]
%
%As the covariant derivative is a linear operator, we immediately have:
%
%\[\pD{}{t}\nabla^p A_{ij} = \nabla^p\nabla_i\nabla_jF - \nabla^p F A_{ik}h^k_j.\]
Using \eqref{EvCurv} and \eqref{eqEvIndMetric} we may compute the evolution of the
mean curvature:
\[
\pD{}{t}H
 = g^{ij}\pD{}{t}A_{ij}
  + A_{ij}\pD{}{t}g^{ij}
 =   \Delta F + F |A|^2
\]
and the tracefree second fundamental form:
\begin{align*}
\pD{}{t}A^o_{ij}
 &= \pD{}{t}A_{ij} - \frac{1}{n}H\pD{}{t}g_{ij} - \frac{1}{n}g_{ij}\pD{}{t}H
\\
 &=
     \nabla_{ij}F - F A_{ik}A^k_j
   + \frac{2}{n}FHA_{ij} - \frac{1}{n}g_{ij}
                           (\Delta F + F |A|^2)
\\
 &=
     \Big(\nabla_{ij}F-\frac{1}{n}g_{ij}\Delta F\Big)
   - FA_{ik}A^k_j
   + F\frac{2}{n}HA_{ij} - F\frac{1}{n}g_{ij}|A|^2
\\
 &=  S^o(\nabla_{(2)}F)
   - F\Big(A_{ik}A^k_j
   -       \frac{2}{n}HA_{ij} + \frac{1}{n}g_{ij}|A|^2\Big)
\\
 &=  S^o(\nabla_{(2)}F)
   - F\Big( (A^o)_{ik}(A^o)^k_j + \frac{1}{n}g_{ij}|A^o|^2 \Big).
\end{align*}
We finally compute the general structure of the evolution of the Christoffel symbols.  Note that any derivative of the
Christoffel symbols is a tensor.
The Christoffel symbols in a local torsion free coordinate system are determined by the metric as
\begin{equation*}
\Gamma_{ij}^k = \frac{1}{2}g^{kl}\left(\pD{}{x^i}g_{jl} + \pD{}{x^j}g_{il} - \pD{}{x^l}g_{ij}\right).
\end{equation*}
Let us choose normal coordinates and differentiate with the help of \eqref{eqEvIndMetric}:
\begin{align*}
\pD{}{t}\Gamma_{ij}^k &= \frac{1}{2}\pD{g^{kl}}{t}\left(\pD{}{x^i}g_{jl} + \pD{}{x^j}g_{il} - \pD{}{x^l}g_{ij}\right)
        + \frac{1}{2}g^{kl}\left(\pD{}{x^i}\pD{}{t}g_{jl} + \pD{}{x^j}\pD{}{t}g_{il} - \pD{}{x^l}\pD{}{t}g_{ij}\right)\\
 &= FA^{kl}\left(\pD{}{x^i}g_{jl} + \pD{}{x^j}g_{il} - \pD{}{x^l}g_{ij}\right)
\\
&\quad
 + g^{kl}\left(-\pD{F}{x^i}A_{jl}-\pD{F}{x^j}A_{il}+\pD{F}{x^l}A_{ij}
              -F\left(\pD{}{x^i}A_{jl}+\pD{}{x^j}A_{il}-\pD{}{x^l}A_{ij}\right)\right)\\
&= FP_1^1(A) + A*\nabla F.
\end{align*}
\end{proof}%}}}

%\begin{rmk}
%Although we do not need it, we note that in two dimensions we have the rather striking evolution equation
%\[
%  \pD{}{t}A^o_{ij} =  S^o(\nabla_{(2)}F)
%   - \frac{1}{2}Fg_{ij}|A^o|^2.
%\]
%This can be proven by taking a normal coordinate basis and checking the equality for each $i$ and $j$ as in the proof of
%\eqref{EQapp6} below.
%\end{rmk}

Interchange of covariant derivatives is used throughout this paper.
The precise consequences used are contained in the following.

\begin{lem}
\label{LMappevoiter}
For $f:M^n\times[0,T)\rightarrow\R^{n+1}$ evolving by $\pD{}{t}f = -(\Delta H + F)\nu$ the following equation
holds:
\[
  \pD{}{t}\nabla_{(k)}A = -\Delta^2\nabla_{(k)}A  + P_3^{k+2}(A) + \nabla_{(k)}\big(FA*A - \nabla_{(2)}F\big).
\]
\end{lem}
\begin{proof}%{{{
Applying \eqref{EQinterchangegeneral} we have
\begin{align*}
\nabla_{ijkl}H
 &= \nabla_{ikjl}H + \nabla\big( A*A*\nabla A \big)
 = \nabla_{kijl}H + A*A*\nabla_{(2)} A + P_3^2(A)
\\
 &= \nabla_{kilj}H + \nabla_{(2)}\big(A*A*A\big) + P_3^2(A)
 = \nabla_{klij}H + P_3^2(A),
\end{align*}
which implies
\[
\nabla_{(2)}\Delta H = \Delta\nabla_{(2)}H + P_3^2(A) 
=
\Delta^2 A + P_3^2(A),
\]
where for the last equality we used Simons' identity \eqref{EQsi}.
Computing in normal coordinates, from Lemma \ref{LMappevo} and the above equation we have
\begin{align*}
   \pD{}{t}\nabla_{(k)}A_{ij}
 &= -\nabla_{(k)}\nabla_{ij}\Delta H - \nabla_{(k)}\nabla_{ij}F  + P_3^{k+2}(A) + \nabla_{(k)}\big(FA*A\big)
\\
 &= -\nabla_{(k)}\Delta A + P_3^{k+2}(A) + \nabla_{(k)}\big(FP_2^0(A) - \nabla_{(2)}F\big).
\end{align*}
The lemma now follows after interchanging covariant derivatives $4k$ times in the term of highest order:
\begin{align*}
     \pD{}{t}\nabla_{(k)}A &- \nabla_{(k)}\big(FA*A - \nabla_{(2)}F\big)
  = -\nabla_{(k)}\nabla^p\nabla_p\Delta A + P_3^{k+2}(A)
\\
 &= -\nabla^p\nabla_{(k)}\nabla_p\Delta A
    + P_3^{k+2}(A)
    + \sum_{j=1}^k\nabla_{(k-j)}\big(A*A*\nabla_{(j)}\Delta A\big)
\\
 &= -\Delta\nabla_{(k)}\Delta A
    + P_3^{k+2}(A)
    + \sum_{j=1}^k\nabla_{(k-j+1)}\big(A*A*\nabla_{(j-1)}\Delta A\big)
\\
 &= -\Delta\nabla^p\nabla_{(k)}\nabla_p A
    + P_3^{k+2}(A) 
    + \sum_{j=1}^k\nabla_{(k-j+2)}\big(A*A*\nabla_{(j-1)}\nabla A\big)
\\
 &= -\Delta^2\nabla_{(k)} A
    + P_3^{k+2}(A)
    + \sum_{j=1}^k\nabla_{(k-j+3)}\big(A*A*\nabla_{(j-1)}A\big)
\\
 &= -\Delta^2\nabla_{(k)} A
    + P_3^{k+2}(A).
\end{align*}
\end{proof}%}}}

\begin{proof}[Proof of \eqref{EQintgradHgradAoapp}]%{{{
We first begin with the following consequence of Simons' identity:
\begin{equation}
\label{EQapp5}
\Delta A^o = S^o(\nabla_{(2)}H) + \frac12H^2A^o - |A^o|^2A^o.
\end{equation}
This follows readily from \eqref{EQsi},
\begin{align*}
\Delta A^o_{ij} &= \Delta A_{ij} - \frac12g_{ij}H
\\
                &= \nabla_{ij}H - \frac12g_{ij}\Delta H + HA_i^jA_{lj} - |A|^2A_{ij}
\\
                &= S^o(\nabla_{(2)}H) + HA_i^jA_{lj} - |A|^2A_{ij}
\\
                &= S^o(\nabla_{(2)}H) + 2KA^o_{ij},
\end{align*}
provided we show
\begin{equation}
\label{EQapp6}
HA_i^jA_{lj} - |A|^2A_{ij} = 2KA^o_{ij}.
\end{equation}
Choosing normal coordinates so that $A$ is diagonalised (at a point) with $A=\delta_i^jk_j$, $H=k_1+k_2$,
$K=k_1k_2$, $|A|^2=k_2^2+k_1^2$ (at this point) and
\[
A^o_{ij} = \begin{cases} \frac12(k_1-k_2),\quad\text{ if $i=j=1$}\\
                         \frac12(k_2-k_1),\quad\text{ if $i=j=2$}\\
                      0,\quad\text{ otherwise},
\end{cases}
\]
we have
\begin{align*}
(i=j=1)\qquad\qquad HA_1^jA_{l1} - |A|^2A_{i1}
 &= (k_1+k_2)k_1^2 - (k_1^2+k_2^2)k_1
\\
 &= \frac12k_1k_2(k_1-k_2)
 = 2KA^o_{11}
\\
(i=j=2)\qquad\qquad HA_2^jA_{l2} - |A|^2A_{i2}
 &= (k_1+k_2)k_2^2 - (k_1^2+k_2^2)k_2
\\
 &= \frac12k_1k_2(k_2-k_1)
 = 2KA^o_{22},
\end{align*}
and otherwise \eqref{EQapp6} holds trivially.  Therefore \eqref{EQapp6} is proved.
This also proves \eqref{EQapp5}, since
\[
K = k_1k_2 =  \frac14(k_1+k_2)^2-\frac14(k_1-k_2)^2 = \frac14H^2-\frac12|A^o|^2.
\]
Integrating \eqref{EQapp5} against $A^o$ and using the divergence theorem we have
\begin{align*}
\int_\Sigma |\nabla A^o|^2d\mu
 &= -\int_\Sigma \IP{A^o}{\Delta A^o}_gd\mu
\\
 &= -\int_\Sigma \IP{A^o}{\nabla_{(2)}H + \frac12H^2A^o - |A^o|^2A^o}_gd\mu
\\
 &= \int_\Sigma \IP{\nabla^*A^o}{\nabla H}_gd\mu
    - \frac12\int_\Sigma H^2|A^o|^2d\mu
    + \int_\Sigma |A^o|^4d\mu.
\end{align*}
Applying \eqref{EQbasicgradHgradAo} and rearranging we obtain
\begin{align*}
\int_\Sigma &|\nabla A^o|^2d\mu
    + \frac12\int_\Sigma H^2|A^o|^2d\mu
  = \frac12\int_\Sigma |\nabla H|^2d\mu
    + \int_\Sigma |A^o|^4d\mu
\end{align*}
as required.
\end{proof}%}}}

%}}}

\bibliographystyle{abbrv}
\bibliography{finitetime}
\end{document}